\documentclass[12pt,a4paper]{article}

\usepackage[english]{babel}
\usepackage[T1]{fontenc}
\usepackage[utf8]{inputenc}

\usepackage{bm}
\usepackage{lmodern}
\usepackage[normalem]{ulem}

\usepackage[top=2.25cm, bottom=2.75cm, left=2.75cm, right=2.75cm]{geometry}
\usepackage{latexsym}
\usepackage{pdflscape}
\usepackage{fancyhdr}
\usepackage{setspace}

\usepackage{graphicx}
\usepackage[dvipsnames]{xcolor}
\usepackage{epstopdf}
\usepackage{epsf}
\usepackage{eepic}
\usepackage{tikz}
\usetikzlibrary{shadings,shapes,arrows,calc,positioning,shadings,decorations.markings, patterns}

\usepackage[numbers]{natbib}

\usepackage{amsmath,amssymb,amsfonts,amscd}
\usepackage{enumitem}
\usepackage{caption}
\usepackage[format=hang,margin=10pt]{subcaption}
\usepackage{listings}
\usepackage{varwidth}
\usepackage{mathrsfs}

\usepackage{algorithm}
\usepackage{algpseudocode}
\usepackage[numbered,framed]{matlab-prettifier}
\usepackage{amsthm}
\usepackage[linguistics]{forest}
\forestset{
nice empty nodes/.style={
    for tree={l sep*=2},
    delay={where content={}{shape=coordinate}{}}
},
}

\usepackage{changepage}
\usepackage{xifthen}
\usepackage{todonotes}

\usepackage{hyperref}
\usepackage{cleveref}

\numberwithin{equation}{section}
\setlist[itemize,1]{label=$\bullet$}
\setlist[itemize,2]{label=$\triangleleft$}
\setlist[enumerate,1]{label=(\roman*)}
\setlist[enumerate,2]{label=(\arabic*)}

\definecolor{TUIl-orange}{RGB}{255, 121, 0}
\definecolor{TUIl-titleblue}{RGB}{0, 68, 121}
\definecolor{TUIl-textblue}{RGB}{0, 51, 88}
\definecolor{TUIl-green}{RGB}{0, 116, 122}
\definecolor{TUIl-grey}{RGB}{165, 165, 165}

\hypersetup{
    breaklinks=true,
    unicode=true,
    pdfpagelayout=OneColumn,
    bookmarksnumbered=true,
    bookmarksopen=true,
    bookmarksopenlevel=0,
    pdfborder={0 0 0},  
    colorlinks=true,
    linkcolor=Cerulean,    
    citecolor=Cerulean,
}

\algnewcommand\algorithmicinput{\textbf{Input:}}
\algnewcommand\AlgInput{\item[\algorithmicinput]}
\algnewcommand\algorithmicoutput{\textbf{Output:}}
\algnewcommand\AlgOutput{\item[\algorithmicoutput]}

\lstMakeShortInline"
\newtheoremstyle{dotless}{}{}{\itshape}{}{\bfseries}{}{ }{}
\newtheoremstyle{no-italic}{}{}{}{}{\bfseries}{}{ }{}
\theoremstyle{dotless}
\newtheorem{Theorem}{Theorem}[section]
\newtheorem{Example}[Theorem]{Example}
\newtheorem{Lemma}[Theorem]{Lemma}
\newtheorem{Definition}[Theorem]{Definition}
\newtheorem{Assumption}{Assumption}
\newtheorem{Remark}[Theorem]{Remark}
\newtheorem{Proposition}[Theorem]{Proposition}
\newtheorem{Corollary}[Theorem]{Corollary}
\newtheorem{Test Instance}[Theorem]{Test Instance}
\theoremstyle{no-italic}

\newcommand*{\R}{\mathbb{R}}

\DeclareMathOperator{\cl}{cl}			% closure
\usepackage{tikz}

\def\R{{\mathbb R}}

\def\dom{\textup{dom }}
\def\WMin{\textup{WMin}}
\def\WMax{\textup{WMax}}
\def\Max{\mathop{\rm Max}}
\def\Min{\mathop{\rm Min}}
\def\SMin{\mathop{\rm SMin}}
\def\Min{\textup{Min}}
\def\epi{\textup{epi }}
\def\Int{\textup{int }}
\def\bd{\textup{bd }}
\def\cl{\textup{cl }}
\def\conv{\textup{conv }}
\def\gph{\textup{gph }}
\def\epi{\textup{epi }}
\title{The Fermat Rule for Set Optimization Problems with Lipschitzian Set-Valued Mappings}
\author{Gemayqzel Bouza  \thanks{Faculty of Mathematics and Computer Science, University of Havana, 10400  Havana, Cuba,
{\texttt{gema@matcom.uh.cu}}} \and Ernest Quintana \thanks{Institute for Mathematics, Technische Universität Ilmenau, 98693 Ilmenau, Germany,
{\texttt{ernest.quintana-aparicio@tu-ilmenau.de}}} \and Christiane Tammer \thanks{Institute of Mathematics, Martin-Luther-Universität Halle-Wittenberg, 06126 Halle, Germany, {\texttt{\{christiane.tammer, anh.vu\}@mathematik.uni-halle.de}}} \and Vu Anh Tuan\footnotemark[3]}
\date{}

\begin{document}

%%%%%%%%%%%%%%%%
%% Title page %%
%%%%%%%%%%%%%%%%
\maketitle

%%%%%%%%%%%%%%
%% Abstract %%
%%%%%%%%%%%%%%
\begin{abstract}
In this paper, we consider set optimization problems where the solution concept is given by the set approach. Specifically, we deal with the lower less and the upper less set relations. First, we derive the convexity and Lipschitzianity of suitable scalarizing functionals under the assumption that the set-valued objective mapping has certain convexity and Lipschitzianity properties. Then, we obtain upper estimates of the limiting subdifferential of these functionals. These results, together with the properties of the scalarization functionals, allow us to obtain a Fermat rule for set optimization problems with Lipschitzian data.
\end{abstract}

%%%%%%%%%%%%%%%%%%%%%%%%%%%%%%%%%%
%% Key words and classification %%
%%%%%%%%%%%%%%%%%%%%%%%%%%%%%%%%%%
\noindent {\small\textbf{Key Words:} set optimization, robust vector optimization, descent method, stationary point}

\vspace{2ex} \noindent {\small\textbf{Mathematics subject
classifications (MSC 2010):}}  	49J53, 90C26, 90C29, 90C48	

%%%%%%%%%%%%%
%% Content %%
%%%%%%%%%%%%%

\section{Introduction}

Set optimization is a class of mathematical problems that consist in minimizing a given set-valued objective mapping. This type of problems generalizes vector optimization models and has received a lot of attention during the last decade due to their applications in finance \cite{FeinsteinZacharyRudloff2015, HamelHeydeLohneRudloffSchrage2015}, socio-economics \cite{Bao9,NeukelNorman2013}, robotics \cite{jahn2011} and robust multiobjective decision making \cite{EhrgottIdeSchobel2014, IdeKobisKuroiwa2014}. 

There are two main approaches for defining optimal solutions of a set optimization problem, namely  the vector approach and the set approach. In the vector approach, we look for efficient points of the image set of the set-valued mapping \cite{jahn2011}. Hence, in this case, one element completely determines the quality of a given set, while ignoring the rest of its elements. This is an important drawback for modeling practical problems, and the set approach is an attempt at fixing this problem. The idea in the set approach is to introduce a preorder relation on the power set of the image space, and to define minimal solutions accordingly. The first set relations were introduced independently by Young \cite{Young1931} and Nishnianidze \cite{Nishnianidze1984}, and later by Kuroiwa \cite{kuroiwa1998, Kuroiwa2001}. More recently, new ones were derived by Jahn and Ha \cite{jahnha2011}, and Karaman et al. \cite{karamanemrahetal2018}. Since then, there have been a lot of research related to the existence of solutions, duality statements, optimality conditions and algorithms for solving these set optimization problems. We refer the reader to \cite{KTZ} and the references therein for a comprehensive overview of the field.

In this paper, we are concerned with necessary optimality conditions for solutions to set optimization problems given by the set approach. The literature on the topic is rich and different results have been obtained using generalized
differentiation objects lying in the primal as well as in the dual spaces (see Giannessi \cite{Gia80}). 

The techniques employed in the primal space are mainly based on some type of directional derivatives and can be roughly separated into the following classes:

\begin{itemize}

\item Directional derivatives based on set differences \cite{dempepilecka2016,jahn2015, karaman2019, pilecka2014}.

The main idea is to consider a suitable operation that resembles subtraction in the power of the image space. These operations are based on the well known differences of sets of Minkowski and Demyanov  \cite{Hadwiger1950, rubinov1992}, but usually slight modifications are introduced in order to make it useful in set optimization. Then, with the help of the set difference, a directional derivative is defined as a limit of an associated incremental quotient. Furthermore, the optimality conditions obtained in this setting establish the nonnegativity of the directional derivative, according to the treated set relation.

\item Directional derivatives based on a distance type functional \cite{ha2018,ha2019}.

In contrast to the previous technique, a directional derivative is introduced in \cite{ha2018} with the help of the standard algebraic difference of sets and a distance type functional. The distance functional is a modification of the well known Hausdorff distance for sets and is based on the classical Hiriart-Urruty functional \cite{Hiriart1}. The directional derivative is in this case defined as the minimal set of some compact set to which the incremental quotient converges (in the sense of the modified distance). A similar idea is used in \cite{ha2019} to introduce a concept of slope for a set-valued map at a given point, together with necessary conditions for minimal solutions of the set optimization problem in the convex case.

\item Directional derivative based on embedding \cite{Kuroiwa2009}.

The idea in \cite{Kuroiwa2009} is to embed the class of convex and bounded sets (with respect to the ordering cone) into a suitable normed space. With this construction, the original set optimization problem is equivalent to a standard vector optimization problem having as a target function the composition of the embedding map and the set-valued objective mapping. Hence, a directional derivative of set-valued mapping is defined in a standard way as the directional derivative of this composition.

\item Directional derivatives of selections of the set-valued objective mapping \cite{Alonsomarin2005, AlonsoMarin2009}.

In this approach, there is no explicit definition of directional derivatives for a set-valued mapping, but rather the authors use those of its continuous selections. Roughly speaking, the optimality conditions establish the nonnegativity, in the sense of the ordering cone, of these directional derivatives.

\item Contingent derivatives and variations \cite{KKY2017, kongetal2017, oussarhandaidai2018, rodmarinsama2007}.

Contingent derivatives and epiderivatives have been successfully employed in obtaining optimality conditions for set optimization problems with  the vector approach \cite{jahn2011}. Consequently, it was a natural idea to apply them also in the set approach setting. In this direction, other modifications of the derivatives were also studied, like those of Shi \cite{Shi1991} and Studniarski \cite{Studniarski-1986}.

\end{itemize}

On the other hand, optimality conditions using generalized
differentiation objects lying in the dual space have been considered in the literature, see \cite{Jahn2017, KTTY2017, KKY2017}. In particular:

\begin{itemize}
\item In \cite{Jahn2017}, the case in which the set-valued mapping is given by functional constraints was analyzed. Using a vectorization  result by Jahn \cite{Jahn2015A}, the set-valued problem was transformed into a vector-valued one (with an infinite dimensional image space), and hence classical optimality conditions for vector optimization problems were applied. 

\item In \cite{KTTY2017, KKY2017} the idea is that, under different assumptions, set approach solutions of the set-valued problem are also solutions in the vector approach. Hence, under these conditions, well studied optimality conditions for vector approach solutions can be applied in the context of the set approach.
\end{itemize}

However, we want to mention that some of these optimality conditions are derived under somewhat strong assumptions on the set-valued objective mapping. For example, in \cite{Alonsomarin2005, AlonsoMarin2009, KKY2017, kongetal2017, oussarhandaidai2018, rodmarinsama2007}, it is required that the optimal set has a strongly minimal element in order to verify optimality. In addition, either the convexity or compactness (mostly both) of the images of the set-valued objective mapping are needed in \cite{dempepilecka2016, ha2018, ha2019, jahn2015, Jahn2017, pilecka2014}. 

\bigskip

Recently, it also caught our attention that, independently, Amahroq and Oussarhan in \cite{amahroqoussarhan2019,amahroq2019convex,amahroqoussarhanPhD2019} and  Huerga, Jim{\'e}nez and Novo in \cite{HueJimNov2020} were working with similar ideas to ours for deriving optimality conditions in set optimization. The main differences between the results derived in these papers and our optimality conditions are the following:

\begin{itemize}
\item In \cite{amahroqoussarhan2019,amahroq2019convex}, \cite{HueJimNov2020}, the authors studied only solution concepts based on the lower less relation. In our paper, we also examine solution concepts based on the upper less relation.

\item In \cite{amahroq2019convex}, \cite{HueJimNov2020}, the case in which the set-valued mapping is convex was analyzed under different assumptions to ours. In addition, the optimality conditions in \cite{amahroq2019convex} require that the optimal set has a strongly minimal element.

\item In \cite{amahroqoussarhan2019}, the case in which the set-valued objective mapping  is locally Lipschitzian is analyzed. However, the authors assume the compactness of the images of the set valued objective mapping in \cite{amahroqoussarhan2019}. The optimality conditions are not established using the initial data, but rather they are expressed  in a limiting form. Also in \cite{HueJimNov2020}, certain compactness assumptions concerning the involved set-valued mappings are supposed.  In our paper, we derive our results without compactness assumptions concerning the set-valued objective mapping.
\end{itemize}

In this paper, as mentioned above, we deal with the lower less relation and the upper less relation, and obtain optimality conditions using generalized
differentiation objects lying in the dual spaces. By means of a suitable scalarizing functional, we construct a  scalar problem that characterizes the solutions of the set-valued problem. Then, based on the initial data, the necessary conditions for the scalar problem are obtained by using well known results from variational analysis. Our results extend those in \cite{DuttaTammer2006} for vector optimization problems.

The paper is organized as follows: In Section \ref{sec: preliminaries}, we introduce fundamental notations, definitions and auxiliary results that will be used through the text. In Section \ref{sec: scalarizing function}, we derive the convexity and Lipschitzianity of suitable scalarizing functionals under certain convexity and Lipschitzianity assumptions on the set-valued objective mapping. Sections \ref{sec: subdifferential l} and \ref{sec: subdifferential u} are devoted to obtaining upper estimates of the limiting subdifferential of these scalarizing functionals. The previous results are employed in Section \ref{sec: opt cond} to derive the optimality conditions for solutions to set optimization problems. Finally, we close the paper  by summarizing our contributions and establishing some further remarks in Section \ref{sec: conclussions}.

\section{Preliminaries}\label{sec: preliminaries}
We start this section by establishing the main notations used in the paper. Given a normed space $(X, \|\cdot\|_X),$ we will denote by $(X^*,\|\cdot\|_X^*)$ its topological dual. In addition, the closed unit balls in $X$ and $X^*$ will be denoted as $\mathbb{B}_X$ and $\mathbb{B}^*_X,$ respectively. We omit the subscript $X$  if there is no risk of confusion. For a nonempty set $A\subseteq X$, $\Int A$, $\cl A$, $\bd A$, $\conv A$ stand for the interior, closure, boundary and convex hull of $A$. If $B \subseteq X^*,$ we denote by $\overline{\conv}^* B$ the closure of the convex hull of $B$ in the weak$^*$ topology of $X^*$.  We always use lowercase letters to denote a vector or scalar-valued function, and capital letters for a set-valued mapping. 

\begin{Definition}Let $X$ be a normed space. Then:

\begin{enumerate}

\item  A nonempty set $C\subseteq X$ is said to be a cone\index{cone} if $\lambda c\in C$ for every $c\in C$ and every $\lambda\geq 0$. The cone $C$ is called:

\begin{itemize}
		\item  convex\index{convex set} if $C+C\subseteq C$,
		\item  proper\index{proper set} if $C\neq\{0\}$ and $C\neq X$, 
		\item  solid\index{reproducing} if $\Int C\neq \emptyset$,
		\item  pointed\index{pointed} if $C\cap (-C)=\{0\}$.
\end{itemize}
	
\item For a cone $C\subseteq X,$ the continuous dual cone of $C$ is given by
\[
C^{*}:=\{x^* \in X^* \mid \forall\; c\in C:\langle x^*,c \rangle\geq 0\}.
\]

\end{enumerate}

\end{Definition}

The support function $\sigma_A: X \rightarrow \overline{\R}$ 
of a set $A\subseteq X^*$ is defined by

\begin{equation*}
\sigma_A(x):=\sup\limits_{x^*\in A}\left\langle x^{*},x \right\rangle, \qquad (x\in X).
\end{equation*}

From now on, we work with the following assumption:

\begin{Assumption}\label{ass 1}
Let $X$ and $Y$ be Banach spaces, $K\subseteq Y$ be a closed, convex, pointed and solid cone and $e\in \Int K$. Let $\Omega \subseteq X$ be nonempty and closed, and fix $\bar{x} \in \Omega.$ Furthermore, let $F:X \rightrightarrows Y$ be a set-valued mapping such that $\Omega \subseteq \Int\dom F.$

\end{Assumption}

Of course, in Assumption \ref{ass 1} above, 

$$\dom F: = \{x\in X \mid F(x)\neq \emptyset\}.$$ Furthermore, the graph and epigraph of $F$ are defined respectively as 

\[
\gph F:=\big\{(x,y)\in X\times Y
\mid y\in F(x)\big\}, 
\]
\[
\epi F:=\big\{(x,y)\in 
X\times Y \mid y\in F(x)+K\big\}.
\]  We define the epigraphical and hypergraphical multifunctions associated to $F$ respectively as the set-valued mappings  $\mathcal{E}_F, \mathcal{H}_F :X\rightrightarrows Y$ given by 

\begin{equation}\label{eq: epi function}
\mathcal{E}_F(x):= F(x)+K,
\end{equation}

\begin{equation}\label{eq: hyper function}
\mathcal{H}_F(x):= F(x)-K.
\end{equation}

The cone $K$ generates a partial order $\preceq_K$ on $Y$ as follows: $y_1\preceq_K y_2$ if and only if $y_2-y_1\in K.$ Associated to $\preceq_K$ is the strict inequality $\prec_K$ which is defined as: $y_1\prec_K y_2$ if and only if $y_2-y_1\in \Int K.$ We also recall that if $y_1,y_2\in Y$ and $y_1\preceq_K y_2$, the interval $[y_1,y_2]$ is defined by
\[
[y_1,y_2]:=(y_1+K)\cap (y_2-K).
\] 

\begin{Definition}

Let Assumption \ref{ass 1}  be fulfilled  and let $A\subseteq Y.$ 

\begin{enumerate}
\item The set of weakly minimal elements of $A$ with respect to $K$ is defined as

$$\WMin (A,K):= \{y \in A\mid \left(y- \Int K\right)\cap A =\emptyset\}.$$

\item  The set of minimal elements of $A$ with respect to $K$ is defined as

$$\Min (A,K):= \{y \in A\mid \left(y-  K\right)\cap A =\{y\}\}.$$
\item The set of weakly maximal elements of $A$ with respect to $K$ is defined as

$$\WMax (A,K):= \{y \in A\mid \left(y+\Int K\right)\cap A =\emptyset\}.$$

\item  The set of maximal elements of $A$ with respect to $K$ is defined as

$$\Max (A,K):= \{y \in A\mid \left(y+  K\right)\cap A =\{y\}\}.$$

\item The set of strongly minimal elements of $A$ with respect to $K$ is defined as

$$\SMin (A,K): = \{y \in A \mid A\subseteq y+K\}.$$

\end{enumerate} 

\end{Definition}

%\tu{Because $e\in \Int K$, the $K$-lower (upper) boundedness of $A$ is equivalent to the existence of a constant $t>0$ such that $ A\subseteq -t U_Y+ K \;\; (\textrm{ respectively, } A\subseteq t U_Y - K).$ Indeed, since $e\in\Int K$, there is $\mu'>0$ such that $-e+\mu'U_Y\subseteq -K$. It follows that $U_Y\subseteq \frac{1}{\mu'}e-K$, and so $\mu U_Y-K\subseteq\frac{\mu}{\mu'}e-K:=te-K$, the converse assertion is clear.	
%}

As mentioned in the introduction, set optimization problems in our context are based on some preorder relations between subsets of $Y.$ These preorders are defined below.

\begin{Definition}[\cite{kuroiwa1998}, \cite{KTH}]\label{def-set-relation}
Let Assumption \ref{ass 1}  be fulfilled and suppose $A,B,D \subseteq Y.$

\begin{enumerate}
	\item The lower- less relation $\preceq^{(l)}_D$ is defined as 

$$A\preceq^{(l)}_D B: \Longleftrightarrow B\subseteq A+D.$$ 

\item The upper- less relation $\preceq^{(u)}_D$ is defined as 

$$A\preceq^{(u)}_D B: \Longleftrightarrow A\subseteq B-D.$$ 

%\item The possibly- less relation $\preceq^{(p)}_D$ is defined as 
%
%$$A\preceq^{(p)}_D B: \Longleftrightarrow B \cap (A+D) \neq \emptyset.$$
%
%\item The certainly- less relation $\preceq^{(c)}_D$ is defined as 
%
%$$A\preceq^{(c)}_D B: \Longleftrightarrow B\subseteq \bigcap_{a \in A} (a+D).$$

\end{enumerate}
\end{Definition}
%They are called in general $r$-relation for $r\in \{l,u,c,p\}$. It follows directly from the definitions that
%\begin{center}
%	\begin{tabular}{cccc}%\label{tab:2}       
%		% Give a unique label
%		$A\preceq_D^{(c)} B$& $\Longrightarrow $ & $A\preceq_D^{(l)} B$ \\ 
%		[2mm] $\Big\Downarrow $ &  &$\Big\Downarrow $ \\ 
%		[2mm] $A\preceq_D^{(u)} B$& $\Longrightarrow $ & $A\preceq_D^{(p)} B$. \\
%		
%	\end{tabular}
%\end{center}

When we take the order with respect to $\Int D,$ we write $\prec_D^{(r)}$ instead of $\preceq_{\Int D}^{(r)},$ for $r \in \{l,u\}.$ Furthermore, for $\preceq^{(l)}_D$ and $\preceq^{(u)}_D$, we recall the equivalence relations on $2^Y$ with respect to a set $D\subseteq Y$ as follows:  

$$A \sim^{(l)}_D B \Longleftrightarrow A \preceq^{(l)}_D B \textrm{ and } B \preceq^{(l)}_D A \Longleftrightarrow A+D = B+D,$$

$$A \sim^{(u)}_D B \Longleftrightarrow A \preceq^{(u)}_D B \textrm{ and } B \preceq^{(u)}_D A \Longleftrightarrow A-D = B-D.$$ 

These family of equivalence classes were first introduced by Hern{\'a}ndez and Rodr{\'{\i}}guez-Mar{\'{\i}}n \cite{HR2007}. For a set $A\subseteq Y,$ we will denote the corresponding equivalence classes by $[A]^{(l)}$ and $[A]^{(u)}$ respectively, depending on the set relation. Under our assumption that $K$ is a closed, convex, pointed and solid cone we get 
\[
A \sim^{(l)}_K B\Longleftrightarrow \Min(A,K)=\Min(B,K),
\]
(compare \cite[Remark 2.6.11]{KTZ}). Next we define convexity of a set-valued mapping with respect to a set relation.

\begin{Definition}
Let Assumption \ref{ass 1}  be fulfilled and let $r \in\{l,u\}$. We say that $F$ is $\preceq_K^{(r)}$-convex if 
	
$$ \forall\; x_1, x_2\in \dom F, \lambda\in (0,1): F(\lambda x_1 + (1-\lambda)x_2)\preceq^{(r)}_K\lambda F(x_1)+(1-\lambda)F(x_2).$$

\end{Definition}
\begin{Remark}
Recall that the classical concept of convexity for a set-valued mapping is that $F:X\rightrightarrows Y$ is convex if its graph is a convex subset of $X\times Y$. It can be shown that $F$ is $\preceq^{(l)}_K$- convex if and only if $\epi F$ is a convex set or, equivalently, if the epigraphical multifunction $\mathcal{E}_F$ is convex. 

\end{Remark}

In the next definition we consider different concepts of boundedness associated to a set-valued mapping.

\begin{Definition}
	
	Let Assumption \ref{ass 1}  be fulfilled and let $U \subseteq  X, \; A\subseteq Y.$ We say that: 
	
	\begin{enumerate}
		
\item $A$ is $K$-lower (upper) bounded if there exists $\mu >0$ such that 

\[
A\subseteq -\mu e+ K \;\; (\textrm{ respectively, } A\subseteq \mu e - K).
\]

\item $F$ is $l$-upper bounded on the set $U$ if there exists a constant $\mu > 0$ such that:
		
		$$\forall\; x\in U: \; F(x)\cap (\mu e - K) \in  [F(x)]^{(l)} .$$
		
\item $F$ is $l$-lower bounded on $U$ if $F[U]$ is $K$-lower bounded. Equivalently, there exists $\mu >0$ such that
		
		$$F[U]\subseteq -\mu e + K.$$
		
\item $F$ is $l$-bounded on $U$ if $F$ is $l$-upper bounded and $l$-lower bounded on $U$. Equivalently, there exists a constant $\mu > 0$ such that:
		
		$$\forall\; x\in U: \; F(x)\cap [-\mu e,\; \mu e] \in  [F(x)]^{(l)}.$$
		
\item $F$ is locally $l$-(upper, lower) bounded at $\bar{x}$ if it is $l$- (upper, lower) bounded on a neighborhood $U$ of $\bar{x}.$
		
\item $F$ is $u$-upper bounded on $U$ if $F[U]$ is $K$- upper bounded. Equivalently, there exists $\mu >0$ such that
		
		$$F[U]\subseteq \mu e - K.$$
		
\item $F$ is $u$-lower bounded on $U$ if there exists a constant $\mu > 0$ such that:
		
		$$\forall\; x\in U: \; F(x)\cap (-\mu e + K) \in [F(x)]^{(u)}.$$
		
\item $F$ is $u$-bounded on $U$ if $F$ is $u$-upper bounded and $u$-lower bounded on $U$, it means that there exists a constant $\mu > 0$ such that:
		
		$$\forall\; x\in U: \; F(x)\cap [-\mu e,\; \mu e]\in [F(x)]^{(u)}.$$
		
\item $F$ is locally $u$-(upper, lower) bounded at $\bar{x}$ if it is $u$- (upper, lower) bounded on a neighborhood $U$ of $\bar{x}.$
		
	\end{enumerate}
	
\end{Definition}

In the following definition, we introduce different topological notions of a set-valued mapping.

\begin{Definition}
	
Let Assumption \ref{ass 1}  be fulfilled. We say that:

\begin{enumerate}

\item $F$ is locally bounded at $\bar{x}$ if there exists $L>0$ and a neighborhood $U$ of $\bar{x}$ such that

$$F[U] \subseteq L \Bbb B.$$ 

\item $F$ is locally Lipschitzian at $\bar{x}$ if there is a neighborhood $U$ of $\bar{x}$ and a constant $\ell\geq 0$ such that

\begin{equation}
\forall\; x,x^{\prime}\in U: \qquad F(x) \subseteq F(x^{\prime})+ \ell\| x-x^{\prime}\|\Bbb B. \label{eq3}%
\end{equation}

%\item $F$ is lower semicontinuous (l.s.c) at $\bar{x}$ if the following   condition is fulfilled:
%
%$$V \textrm{ open, }\;  F(\bar{x})\cap V\neq \emptyset \Longrightarrow \exists \;\epsilon >0: F(x)\cap V\neq \emptyset \; \forall\; x\in \bar{x} + \epsilon \Bbb B.$$ 
%

\item  $F$ is inner semicompact at $\bar{x} \in \dom F$ if for every sequence $x_k \rightarrow \bar{x}$ there is a sequence $y_k \in F(x_k)$ that contains a convergent subsequence as $k \rightarrow \infty$. In particular, $\bar{x} \in \Int \dom F.$

\item $F$ is closed at $\bar{x}$ if, for any sequence $\{(x_k,y_k)\}_{k\geq 1} \subseteq \gph F$ with $(x_k,y_k) \to (\bar{x},\bar{y}),$  we have $(\bar{x},\bar{y})\in \gph F.$
\end{enumerate}

\end{Definition}
%
%\begin{Remark}\label{rem: lipsch lsc}
%It easily follows from the definition that if $F$ is locally Lipschitz at $\bar{x},$ it is also l.s.c at $\bar{x}.$ 
%
%\end{Remark}

For a scalar function $f:X\rightarrow \overline{\R}$, the domain\index{domain} and epigraph of $f$ are given by
\[
\operatorname{dom}f:=\{x\in X\mid f(x)<+\infty\},
\]

\[
\epi f:=\{(x,t)\in 
X\times \R \mid f(x)\leq t\}.
\]
Recall that a function $f: X\rightarrow \R$ is convex if $\epi f$ is a convex set. The function $f$ is said to be Lipschitzian on $A$ provided that $f$ is finite on $A$ and there exists $\ell>0$ such that
\[
	\forall \, x,x' \in A:\;|f(x)-f(x')|\leq \ell \| x-x'\|_X.
\] This is also referred to as a Lipschitzian condition of rank $\ell$. We say that $f$ is locally Lipschitzian at $\bar{x}$ if there is a neighborhood $U$ of $\bar{x}$ such that $f$ is Lipschitzian on $U$. In addition, $f$ is said to be locally Lipschitzian on $A$, if $f$ is locally Lipschitzian at every point $x\in A$. Hence, in this case, $A\subseteq\Int\operatorname*{dom}f.$

We now introduce the tools from variational analysis that will be employed in the text. First, we need a notion of limits of sets. For a set-valued mapping $F:X\rightrightarrows X^{*}$, we define the
Painlev\'{e}-Kuratowski outer limit of $F$ at $\bar{x}$ with respect to the
norm topology of $X$ and the $w^*$- topology of $X^{*}$ by
\[
\underset{x\rightarrow\bar{x}}{\limsup}\;F(x):=\{x^{*}\in X^{*}\mid\forall\;
k\in\mathbb{N},\ \exists\;(x_{k},x_{k}^{*})\in\operatorname*{gph}%
F:x_{k}\rightarrow\bar{x},\text{ }x_{k}^{*}\xrightarrow{w^*}x^{*}\}.
\]

In the next definition, we use the notation $x^{\prime}\overset{\Omega} {\longrightarrow}x$ for $x^{\prime}\rightarrow x$ with $x^{\prime}\in\Omega$.

\begin{Definition} \label{def: lipnorcon}
(\cite[Definition 1.1]{Mordukhovich1}) 
Let Assumption \ref{ass 1}  be fulfilled.
\begin{enumerate}

\item Given $x \in X$ and $\epsilon\geq 0$, the set of $\epsilon$-normals to $\Omega$ at $x$ is defined by 
\begin{equation}
\hat{N}_\epsilon(x,\Omega):=\left\{x^*\in X^*\mid\underset
{u\overset{\Omega}{\longrightarrow}x}{\limsup}\frac{\langle x^{*},u-x\rangle}{\left\Vert
u-x\right\Vert }\leq\epsilon\right\}.
\label{eqnor1}
\end{equation}
When $\epsilon =0$, the set \eqref{eqnor1} is called the
Fr\'{e}chet normal cone to $\Omega$ at $x$, and is denoted by $\hat{N}(x,\Omega)$. If $x\notin\Omega$, we put $\hat{N}_\epsilon(x,\Omega):=\emptyset$ for all $\epsilon\geq 0$.

\item The limiting normal cone to $\Omega$ at $\bar{x}$ is defined by
\begin{equation}
N(\bar{x},\Omega):=\underset{\epsilon\downarrow 0}{\underset{x\rightarrow\bar{x}}{\limsup}}\hat{N}_\epsilon(x,\Omega).
\label{eqnor2}
\end{equation}
We also put $N(\bar{x},\Omega):=\emptyset$ for $\bar{x}\notin \Omega$.
\end{enumerate}
\end{Definition}

\begin{Definition} \label{defmorsub}(\cite[Definition 1.77]
{Mordukhovich1}) Let Assumption \ref{ass 1}  be fulfilled. The limiting subdifferential of a given functional $f:X\rightarrow \mathbb{R}\cup\{\pm\infty\}$ at a point $\bar{x}$ with
$|f(\bar{x})|<+\infty$ is defined by
\[
\partial f(\bar{x}):=\big\{x^{*}\in X^{*}\mid(x^{*},-1)\in
N \big((\bar{x},f(\bar{x})),\operatorname*{epi}f\big)\}.
\]
We put $\partial f(\bar{x}):=\emptyset$ if $|f(\bar{x})|=+\infty$.
\end{Definition}

\begin{Remark}
It is well known, see \cite[Theorem 1.93]{Mordukhovich1}, that if $f$ is convex and finite at $\bar{x}$, then

$$\partial f(\bar{x})=\{x^{*}\in X^*\mid\forall\; x\in X : f(x)-f(\bar{x})\geq \langle x^*, x-\bar{x} \rangle\ \},$$ and hence $\partial f(\bar{x})$ coincides with the subdifferential of convex analysis. In case $\Omega$ is a convex set, we also have \cite[Proposition 1.5]{Mordukhovich1}:

$$N(\bar{x},\Omega)=\{x^*\in X^*\mid \forall\; x\in \Omega:  \langle x^*, x-\bar{x} \rangle\leq 0\},$$ and hence $N(\bar{x},\Omega)$ equals the normal cone in the sense of convex analysis.
\end{Remark}

\begin{Remark}\label{rem: subdif -f}
Note that, if in addition to Assumption \ref{ass 1} the space $X$ is Asplund and $f:X \rightarrow \R$ is locally Lipschitzian at $\bar{x},$ the following relation holds: 

$$\partial(-f)(\bar{x}) \subseteq -\overline{\operatorname{conv}}^* \left(\partial f(\bar{x})\right).$$ 

Indeed, taking into account \cite[\textrm{Theorem 3.57}]{Mordukhovich1}, \cite[\textrm{Proposition 2.3.1}]{Clarke1} and \cite[\textrm{Theorem 3.57}]{Mordukhovich1}, we obtain

\begin{eqnarray*}
\partial(-f)(\bar{x}) & \subseteq &\overline{\operatorname{conv}}^* \left(\partial (-f)(\bar{x})\right)\\
                        &  {=} & \partial^\circ (-f)(\bar{x})\\
                        &  {=} & - \partial^\circ f(\bar{x})\\
                        &  {=} & -\overline{\operatorname{conv}}^* \left(\partial f(\bar{x})\right).
\end{eqnarray*} Here, $\partial^\circ f$ represents Clarke's subdifferential of $f,$ see \cite{Clarke1} for details.
\end{Remark}

We continue by defining the basic coderivative of a set-valued mapping at a point of its graph. 
\begin{Definition} \label{defcord}(\cite[Definition 1.32]
{Mordukhovich1}) Let Assumption \ref{ass 1}  be fulfilled. The basic coderivative of $F$ at $(\bar{x},\bar{y})\in \gph F$ is the multifunction $D^* F(\bar{x},\bar{y}): Y^*\rightrightarrows X^*$ defined by
\begin{equation}
D^{*}F(\bar{x},\bar{y})(y^{*})=\big\{x^{*}\in X^{*}\mid(x^{*
},-y^{*})\in N\big((\bar{x},\bar{y}),\operatorname*{gph}F\big)\}.
\label{cod}%
\end{equation}
We put $D^* F(\bar{x},\bar{y})(y^*):=\emptyset$ for all $y^*\in Y^*$ if $(\bar{x},\bar{y})\notin \gph F$. 

\end{Definition}

\begin{Remark}\label{rem: coderivative of funct}
We can omit $\bar{y}$ in the coderivative notation above
if $F=f: X\rightarrow Y$ is a vector-valued function. It can be shown  \cite[Theorem 1.38]{Mordukhovich1},  that if $f$ is continuously differentiable at $\bar{x}$, then
\[
D^{*}f(\bar{x})(y^{*})=\big\{\nabla  f(\bar{x})^{*}y^{*} \} \quad\text{for all}\quad y^{*}\in Y^{*}.
\] In the above equation, $\nabla  f(\bar{x})^{*}:Y^* \to X^*$ denotes the adjoint operator of $\nabla  f(\bar{x}).$

\end{Remark}

%\begin{Remark}\label{rem: coderivative of functional and subdiff}
%
%Furthermore, if $f$ is strictly Lipschitzianian at
%$\bar{x}$, the relationship between the co\-derivative of a vector function
%and the sub\-differential of its scalarization is given by \cite[Theorem
%3.28]{Mordukhovich1}:
%\[
%D^{*}f(\bar{x})(y^{*})=\partial (y^{*}\circ f)(\bar{x}).
%\]
%
%\end{Remark}

In the last part of the section, we quickly recall two results concerning the subdifferential of marginal functions. This problem is naturally linked to the computation of the subdifferentials of the scalarization functionals, as we will see in the rest of the sections. The setting is as follows:

\begin{Assumption}\label{ass marginal}
In addition to Assumption \ref{ass 1}, let $f:X\times Y \to \overline{\R}$ be a given functional and consider the associated marginal function $\varphi:X\to \overline{\R}$ defined as

\begin{equation*}\label{eq: marginal function}
\varphi(x):= \inf_{y\in F(x)} f(x,y).
\end{equation*} 
Furthermore, consider the solution map $S:X\rightrightarrows Y$ defined as 

\begin{equation*}\label{eq: solution map}
S(x)= \{y\in F(x): f(x,y)= \varphi(x)\}.
\end{equation*} 

\end{Assumption}

The first of the results is concerned with the subdifferential of $\varphi$ in the case that $F$ and $f$ are assumed to be convex. 

\begin{Theorem}(\cite[Theorem 4.2]{AY2015})\label{thm: convex inf functions subdif} Let Assumption \ref{ass marginal} be fulfilled. Suppose in addition that $\gph F$ is convex,  $f$ is a proper and convex function, and that at least one of the following regularity conditions is satisfied: 
	\begin{enumerate}
	\item $\Int\gph F\cap \dom f\neq \emptyset$,  

\item $f$ is continuous at a point $(x^0,y^0)\in \gph F$.
\end{enumerate}
Then, $\varphi$ is convex and, for any $\bar{x} \in \dom \varphi$ with $\varphi(\bar{x})\neq-\infty$ and any $\bar{y}\in S(\bar{x}),$ we have

$$\partial\varphi(\bar{x})= \bigcup_{(x^*,y^*)\in \partial f(\bar{x},\bar{y})}  \bigg[ x^*+ D^* F(\bar{x},\bar{y})(y^*) \bigg].$$

\end{Theorem}

For nonconvex $F$, many results already exist in the literature. We conclude by establishing a weaker version of \cite[Theorem 3.38 $(ii)$]{Mordukhovich1}, which will be enough for our purposes. The proof is omitted since it is easy to verify that our assumptions imply those of \cite[Theorem 3.38 $(ii)$]{Mordukhovich1}. Recall that a set $A$ is locally closed around a point $\bar{z} \in A$ if there exists a neighborhood  $V$ of $\bar{z}$ such that $A\cap V$ is a closed set.

\begin{Theorem}\label{thm: basic subdif marginal functions theorem}
In addition to Assumption \ref{ass marginal}, suppose that $X$ and $Y$ are Asplund spaces. Furthermore, assume that:

\begin{enumerate}

\item $F$ is closed at $\bar{x},$

\item $S$ is inner semicompact at $\bar{x},$

\item there exists a neighborhood $U$ of $\bar{x}$ such that $f$ is Lipschitzian on $U\times Y,$

\item $\gph F$ is locally closed around every point of the set $\{\bar{x}\} \times S(\bar{x}).$

\end{enumerate} Then,
\begin{equation}\label{eq: upper estimate marginal}
\partial \varphi(\bar{x})\subseteq  \bigcup_{\underset{(x^*,y^*)\in \partial f(\bar{x},\bar{y})}{\bar{y}\in S(\bar{x})}}  \bigg[ x^*+ D^* F(\bar{x},\bar{y})(y^*) \bigg].
\end{equation}

\end{Theorem}

\section{Convexity and Lipschitzianity of the Scalarizing Functionals in Set Optimization} \label{sec: scalarizing function}

With the purpose of deriving optimality conditions for set optimization problems we introduce in this section, for a given set-valued mapping $F$, two associated scalarizing functionals. 

We proceed to show that certain nonlinear scalarizing functionals inherit the convexity and Lipschitzianity of $F.$ First, under Assumption \ref{ass 1}, we consider the functional $\Psi_e: Y\rightarrow \bar{\Bbb R}$ given by:

\begin{equation}\label{eq: gerstewitz funct}
\Psi_e(y):=\operatorname{inf}\{t\in\Bbb R\mid y \in te-K\}.
\end{equation}

Recent results concerning characterizations of set relations by scalarizing functionals and corresponding subdifferentials are given in \cite{BaoTam2018,BaoTam2019}. 

The nonlinear scalarizing functional $\Psi_e$ has been widely applied in vector optimization \cite{TW1990,GopfertRiahiTammerZalinescu2003,KTZ} and in the next proposition we collect several well known properties of $\Psi_e$ that will be useful later in this work. Recall that a functional $g:Y \to \R$ is said to be $K$- monotone if $y_1,y_2\in Y,y_1\preceq_K y_2\Rightarrow g(y_1)\leq g(y_2)$. Moreover, we say that $g$ is strictly $K$-monotone if $y_1\prec_K y_2\Rightarrow g(y_1)<g(y_2).$
\begin{Proposition}[\cite{GopfertRiahiTammerZalinescu2003,DT2009}]\label{prop: properties of tammer function} 
Let Assumption \ref{ass 1} be fulfilled. Then:
\begin{enumerate}
		\item $\Psi_e$ is a finite-valued sublinear function, i.e., $\Psi_e(y_1+y_2)\leq \Psi_e(y_1)+\Psi_e(y_2)$ and $\Psi_e(\lambda y)=\lambda \Psi_e(y)$ for all $\lambda > 0$ and $y_1,y_2\in Y$,
		\item $\Psi_e$  is Lipschitzian on $Y,$
		\item  $\Psi_e$ satisfies the translativity property, i.e., $\Psi_e(y+te)=\Psi_e(y)+t$ for all $t\in \Bbb R$ and $y\in Y,$
		\item  $\Psi_e$ is $K$- monotone and strictly $K$-monotone,
		\item  $\partial \Psi_e(\bar{y})=\{ k^*\in K^*\mid \langle k^*,e\rangle =1, \Psi_e(\bar{y})=\langle k^*,\bar{y}\rangle\}, $
		
		\item $\Psi_e$ satisfies the representability property, i.e., 
		
		$$-K= \{y \in Y \mid \Psi_e(y)\leq 0 \}, \quad - \Int K = \{y \in Y \mid \Psi_e(y)< 0 \}.$$
	
\end{enumerate}
\end{Proposition}
%%%%%%%%%%%%%%%%
\begin{Remark}
According to $(vi),$ for any $\bar{y} \in -\bd K$ we have $\Psi_e(\bar{y}) = 0.$ Then, it follows from $(v)$ that $\partial \Psi_e(\bar{y}) = \{k^*\in K^*\mid \langle k^*,e\rangle =1, \langle k^*,\bar{y}\rangle = 0\}.$ This simple fact is important to keep in mind for the results that will be obtained later.

\end{Remark}

In \cite{KoeTam15,KK2016,KKY2016}, a complete characterization of set order relations by means of a nonlinear scalarizing functional was shown. There, the main result is the following:

\begin{Theorem}(\cite{KoeTam15,KK2016,KKY2016})\label{thm: characterization of set relations by scalarizing functionals}
Let Assumption \ref{ass 1}  be fulfilled and consider $A,B \subseteq Y.$ Then,
\begin{enumerate}

\item $A\preceq^{(l)}_K B \Longrightarrow \sup\limits_{b\in B}\inf\limits_{a\in A} \Psi_e(a-b)\leq 0.$

\item $A\preceq^{(u)}_K B \Longrightarrow \sup\limits_{a\in A}\inf\limits_{b\in B} \Psi_e(a-b)\leq 0.$

\end{enumerate}

\end{Theorem} 

The previous theorem motivates our next definition.  

\begin{Definition}\label{def- scalarization} 

Let Assumption \ref{ass 1}  be fulfilled. 

\begin{enumerate}

\item The lower inner function $g_{l}:X \times Y \to  \bar{\Bbb R}$ is defined as 

\begin{equation}\label{lower_function}
  g_{l}(x,z):= \inf_{y\in F(x)} \Psi_e(y-z).
\end{equation}
\item The upper inner function $g_{u,\bar{x}}: Y \to \bar{\Bbb R}$ is defined as

\begin{equation}\label{upper_function}
g_{u,\bar{x}}(y):= \inf_{\bar{y}\in F(\bar{x})}\Psi_e(y-\bar{y}).
\end{equation}
\item For $r\in \{l,u\},$ the scalarizing functional $f_{r,\bar{x}}:X\to \bar{\R}$ is defined as follows: 

\begin{equation}\label{eq f_r}
f_{r,\bar{x}}(x):= \left\{
\begin{array}{ll}
   \sup\limits_{\bar{y}\in F(\bar{x})}g_{l}(x,\bar{y})=\sup\limits_{\bar{y}\in F(\bar{x})}\inf\limits_{y\in F(x)} \Psi_e(y- \bar{y})   & \textrm{ if } r=l, \\
   \sup\limits_{y\in F(x)} g_{u,\bar{x}}(y)=\sup\limits_{y\in F(x)}\inf\limits_{\bar{y}\in F(\bar{x})} \Psi_e(y- \bar{y})   & \textrm{ if } r=u. \\
\end{array} 
\right. 
\end{equation}

\end{enumerate}

\end{Definition}

As mentioned at the begining of the section, we now show that for the $\preceq_K^{(l)}$ and $\preceq_K^{(u)}$ relations, the corresponding scalarizing functional inherits the convexity property of the set-valued mapping. We start with a simple proposition.

\begin{Proposition}\label{prop: g =0}
Let Assumption \ref{ass 1}  be fulfilled and consider the  functionals given in Definition \ref{def- scalarization}. Then, the following statements are true:
\begin{enumerate}
\item For every $x \in X,$ the functional $g_l(x,\cdot)$ is $-K$- monotone. Furthermore, for $\bar{y} \in F(\bar{x}),$ we have that $g_{l}(\bar{x},\bar{y}) = 0$ if and only if $\bar{y} \in \WMin(F(\bar{x}),K).$

\item The functional $g_{u,\bar{x}}$ is $K$-monotone. Furthermore, for $y \in  F(\bar{x}),$ we have that $g_{u,\bar{x}}(y) = 0$ if and only if $y \in \WMax(F(\bar{x}),K).$ 

\item  For any $r\in \{l,u\},$ we have $f_{r,\bar{x}}(\bar{x})\leq 0.$ Equality holds if $r=l$ and $\WMin(F(\bar{x}),K)\neq \emptyset,$ or $r=u$ and $\WMax(F(\bar{x}),K)\neq \emptyset.$
\end{enumerate}
\begin{proof}
$(i)$  The monotonicity of $g_l(x,\cdot)$ follows directly from the monotonicity of $\Psi_e.$ Now, fix $\bar{y} \in F(\bar{x}).$ Then, we have $g_{l}(\bar{x},\bar{y})\leq 0$ and hence

\begin{eqnarray*}
g_{l}(\bar{x},\bar{y})= 0  & \Longleftrightarrow & \inf_{y \in F(\bar{x})} \Psi_e(y-\bar{y})\geq 0\\
                           & \Longleftrightarrow & \forall \; y \in F(\bar{x}): \;  \Psi_e(y-\bar{y})\geq 0\\
                           & \Longleftrightarrow & \forall \; y \in F(\bar{x}): y- \bar{y} \notin - \Int K\\
                           & \Longleftrightarrow & \bar{y} \in \WMin(F(\bar{x}),K).                   
\end{eqnarray*}

$(ii)$ The monotonicity of $g_{u,\bar{x}}$ is easily deduced from the monotonicity of $\Psi_e.$ Now, take $y\in F(\bar{x}).$ Then, we always have $g_{u,\bar{x}}(y)\leq 0.$ Analogous to $(i)$ we get

\begin{eqnarray*}
g_{u,\bar{x}}(y)= 0  & \Longleftrightarrow & \inf_{\bar{y} \in F(\bar{x})} \Psi_e(y-\bar{y})\geq 0\\
                           & \Longleftrightarrow & \forall \; \bar{y} \in F(\bar{x}): \;  \Psi_e(y-\bar{y})\geq 0\\
                           & \Longleftrightarrow & \forall \; \bar{y} \in F(\bar{x}): y- \bar{y} \notin - \Int K\\
                           & \Longleftrightarrow & y \in \WMax(F(\bar{x}),K),                   
\end{eqnarray*} as desired.

$(iii)$ The fact that $f_{r,\bar{x}}(\bar{x})\leq 0$ is trivial. If $r=l$ and $\tilde{y}\in  \WMin(F(\bar{x}),K)\neq \emptyset$ then, by statement $(i)$, we get $g_{l}(\bar{x},\tilde{y})=0.$ From this we deduce that $f_{l,\bar{x}}(\bar{x})\geq g_l(\bar{x},\tilde{y})= 0$, and hence the equality holds. Analogously, if $r=u$ and $\tilde{y}\in  \WMax(F(\bar{x}),K)\neq \emptyset$ then, by statement $(ii)$, we get $g_{u,\bar{x}}(\tilde{y})=0.$ Again, this implies that $f_{u,\bar{x}}(\bar{x})\geq 0,$ and hence the equality.

\end{proof}

\end{Proposition}
\begin{Remark}\label{rem: replace F by epiF}
Proposition \ref{prop: g =0} $(i)$ together with  Proposition \ref{prop: properties of tammer function} $(iv)$ gives us monotonicity properties  of the functionals $g_{l}(x,\cdot)$ and $\Psi_e.$ From this, it easily follows that the functionals $g_{l}$ and $f_{l,\bar{x}}$ are invariant under replacement of $F$ by any set-valued mapping of the form $F_A:= F+ A,$ with $A\subseteq K$ and $0 \in A.$ In particular, this is true when we replace $F$ by $\mathcal{E}_F.$ 

Similarly, from Proposition \ref{prop: g =0} $(ii)$ and Proposition \ref{prop: properties of tammer function} $(iv)$ we deduce that the functionals $f_{u,\bar{x}}$ and $g_{u,\bar{x}}$ are invariant under replacement of $F$ by any set-valued mapping of the form $F_A:= F-A,$ with $A\subseteq K$ and $0 \in A.$

\end{Remark}

The next lemma proves some useful properties of the inner functions in the convex case
\begin{Lemma}\label{lem: convexity inner functions}

Let Assumption \ref{ass 1} be satisfied and consider the lower and upper inner functions given in Definition \ref{def- scalarization}. The following statements hold:

\begin{enumerate}

\item If $F$ is $\preceq^{(l)}_K$-convex, then $g_{l}(\cdot,\bar{y})$ is convex for every $\bar{y}\in F(\bar{x}).$ Furthermore, if $F$ is locally $l$-bounded at $\bar{x},$ then $\bar{x}\in \Int \operatorname{dom }g_{l}(\cdot,\bar{y})$ and $g_{l}(\cdot,\bar{y})$ is continuous at $\bar{x}.$

\item If $\mathcal{H}_F(\bar{x})$ is a convex and $K$- upper bounded set, then $g_{u,\bar{x}}$ is a convex $K$-monotone functional that is continuous on $Y.$

\end{enumerate}

\end{Lemma}
\begin{proof}
$(i)$ Take $\bar{y} \in F(\bar{x}),\; x_1, x_2 \in X,$ and $\lambda \in (0,1).$ Let  $x_\lambda:= \lambda x_1+(1-\lambda) x_2$ and $F_\lambda:=\lambda F(x_1)+(1-\lambda) F(x_2).$ Since $F$ is $\preceq^{(l)}_K$-convex, we have 

\begin{equation}\label{eq: convex l}
F_\lambda\subseteq F(x_\lambda)+ K.
\end{equation} 
We now have 
\begin{eqnarray*}
&g_{l}(\lambda x_1+(1-\lambda) x_2, \bar{y})&\\
 & = & \inf_{y\in F(x_\lambda)} \Psi_e(y-\bar{y})\\
                                           &\overset{(\;\mathrm{by }\;  \mathrm{ monotonicity} \; \mathrm{of}\; \Psi_e\;)}{=} & \inf_{y\in F(x_\lambda)+K} \Psi_e(y-\bar{y})\\
                                           &\overset{(\;\mathrm{by }\;  \eqref{eq: convex l}\;)}{\leq} & \inf_{y\in F_\lambda} \Psi_e(y-\bar{y})\\
                                           & = & \inf_{(y_1,y_2)\in F(x_1)\times F(x_2)} \Psi_e(\lambda y_1+(1-\lambda) y_2-\bar{y})\\
                                           & = & \inf_{(y_1,y_2)\in F(x_1)\times F(x_2)} \Psi_e(\lambda (y_1-\bar{y})+(1-\lambda) (y_2-\bar{y}))\\
                                           & \overset{(\;\mathrm{by }\;  \mathrm{ convexity} \; \mathrm{of}\; \Psi_e\;)}{\leq} & \inf_{(y_1,y_2)\in F(x_1)\times F(x_2)} \lambda \Psi_e (y_1-\bar{y})+ (1-\lambda) \Psi_e (y_2-\bar{y})\\
                                           & = & \lambda g_{l}(x_1,\bar{y})+ (1-\lambda) g_{l}(x_2,\bar{y}).\\
\end{eqnarray*}
Now, let us assume that $F$ is locally $l$-bounded at $\bar{x}.$ Hence, we can find
$ \mu > 0$ and a neighborhood $U$ of $\bar{x}$ such that 

$$\forall\; x\in U: \; F(x)\cap [-\mu e, \; \mu e] + K = F(x) + K.$$
 By the monotonicity of $\Psi_e$, we have, for every $x \in U:$

\begin{eqnarray}\label{eq: g_u,x finite}
 -\infty                       & <    &  \Psi_e(-\mu e -\bar{y}) \nonumber \\
                                   & =    &  \inf_{y \in -\mu e+K} \Psi_e(y-\bar{y})\nonumber\\
                                   & \leq &  \inf_{y\in F(x) +K } \Psi_e(y-\bar{y})\nonumber\\
                                   & = &  g_{l}(x,\bar{y}) \\
                                   & =    &  \inf_{y\in F(x)\cap [\mu e -K]} \Psi_e(y-\bar{y}) \nonumber \\
                                   & \leq &  \sup_{y\in F(x)\cap [\mu e -K]} \Psi_e(y-\bar{y}) \nonumber\\
                                   & \leq &  \Psi_e(\mu e-\bar{y}) \nonumber\\
                                   & <    & +\infty  .          \nonumber                        
\end{eqnarray} This shows that $g_{l}(\cdot,\bar{y})$ is finite and bounded above around $\bar{x},$ from which the continuity is deduced.\\

$(ii)$ The monotonicity of $g_{u,\bar{x}}$ was already established in Proposition \ref{prop: g =0} $(ii)$. In order to show the convexity, we check that $\epi g_{u,\bar{x}}$ is convex. Indeed, take $(y_1,t_1), (y_2,t_2)\in \epi g_{u,\bar{x}}$ and $\lambda \in (0,1).$ Hence, $g_{u,\bar{x}}(x_1)\leq t_1 $ and $g_{u,\bar{x}}(x_2)\leq t_2.$ Then, for any $\epsilon>0,$ we have $$g_{u,\bar{x}}(x_1)< t_1+\epsilon,\;\; g_{u,\bar{x}}(x_2)< t_2+\epsilon. $$ But then, we can find $\bar{y}_1, \bar{y}_2\in F(\bar{x})$ such that $$\Psi_e(y_1- \bar{y}_1) <t_1+\epsilon,\;\; \Psi_e(y_2- \bar{y}_2) <t_2+\epsilon.$$ From this, we get 

\begin{eqnarray*}
\Psi_e((\lambda y_1 +(1-\lambda)y_2 )- (\lambda \bar{y}_1+(1-\lambda)\bar{y}_2)) & = & \Psi_e(\lambda(y_1-\bar{y}_1) + (1-\lambda) (y_2-\bar{y}_2))\\
                                                                                 & \leq & \lambda \Psi_e(y_1- \bar{y}_1) +(1-\lambda) \Psi_e(y_2-\bar{y}_2)\\
                                                                                 & \leq &  \lambda(t_1+\epsilon) +(1-\lambda)(t_2+\epsilon)\\
                                                                                 & = & \lambda t_1+(1-\lambda)t_2 +\epsilon. 
\end{eqnarray*} Now, because $F(\bar{x})-K$ is convex, we have 

$$\lambda \bar{y}_1+(1-\lambda)\bar{y}_2 \in \operatorname{conv}(F(\bar{x}))\subseteq \mathcal{H}_F(\bar{x}),$$ and hence we can find $\bar{y} \in F(\bar{x})$ such that $\lambda \bar{y}_1+(1-\lambda)\bar{y}_2 \in \bar{y}-K.$ Then, by monotonicity of $\Psi_e,$ we get

\begin{eqnarray*}
g_{u,\bar{x}}(\lambda y_1+(1-\lambda)y_2) & \leq & \Psi_e(\lambda y_1+(1-\lambda)y_2 - \bar{y})\\
                                          & \leq & \Psi_e((\lambda y_1+(1-\lambda)y_2)- (\lambda \bar{y}_1+(1-\lambda)\bar{y}_2))\\
                                          & \leq &  \lambda t_1+(1-\lambda)t_2 +\epsilon.
\end{eqnarray*} Since $\epsilon>0$ was chosen arbitrarily, we conclude that $(\lambda y_1+(1-\lambda)y_2, \lambda t_1+(1-\lambda)t_2)\in \epi g_{u,\bar{x}}.$ But this means that $\epi g_{u,\bar{x}}$ is a convex set, as desired. 

Now, since $\mathcal{H}_F(\bar{x})$ is $K$-upper bounded, we have  

$$-\infty < \Psi(y- \mu e) =\inf_{\bar{y}\in \mu e-K} \Psi_e(y-\bar{y}) \leq \inf_{\bar{y}\in \mathcal{H}_F(\bar{x})-K} \Psi_e(y-\bar{y}) \overset{\textrm{(Remark \ref{rem: replace F by epiF})}}{=} g_{u,\bar{x}}(y).$$ This means that $g_{u,\bar{x}}$ is finite on $Y.$ The continuity of $g_{u,\bar{x}}$ is now deduced by fixing $\bar{y}\in F(\bar{x})$ and noticing that $g_{u,\bar{x}}(\cdot) \leq \Psi_e(\cdot -\bar{y}),$ a continuous convex functional.
\end{proof}

We are now ready to establish the convexity of the scalarization functions $f_{l,\bar{x}}$ and $f_{u,\bar{x}}.$

\begin{Theorem}\label{thm: convexity of scalarizations }

Let Assumption \ref{ass 1} be satisfied and, for $r\in \{l,u\},$ consider the functional $f_{r,\bar{x}}$ given in Definition \ref{def- scalarization} $(iii)$. The following statements hold:

\begin{enumerate}
\item If $F$ is $\preceq^{(l)}_K$-convex then $f_{l,\bar{x}}$ is convex. Furthermore, if $F$ is locally $l$-bounded at $\bar{x},$ then $\bar{x} \in \Int \operatorname{dom } f_{l,\bar{x}} $ and $f_{l,\bar{x}}$ is continuous at $\bar{x}.$

\item If $F$ is $\preceq^{(u)}_K$-convex and $\mathcal{H}_F(\bar{x})$ is a convex set, then $f_{u,\bar{x}}$ is convex. Furthermore, if $F$ is locally $u$-upper bounded at $\bar{x},$ then $\bar{x} \in \Int \operatorname{dom } f_{u,\bar{x}} $ and $f_{u,\bar{x}}$ is continuous at $\bar{x}.$ 

\end{enumerate}
\end{Theorem}
\begin{proof}

$(i)$ We have $$f_{l,\bar{x}}(x)= \sup_{\bar{y}\in F(\bar{x})} g_{l}(x,\bar{y}).$$ By Lemma \ref{lem: convexity inner functions} (i), for every $\bar{y}\in F(\bar{x}),$ the functional $g_{l}(\cdot,\bar{y})$ is convex. Hence $f_{l,\bar{x}}$ is convex as it is the supremum of convex functionals. To prove the second part, it suffices to show that $f_{l,\bar{x}}$ is finite and upper bounded on a neighborhood of $\bar{x}.$ In order to show that this is true, note that the assumptions on the second part of Lemma \ref{lem: convexity inner functions} $(i)$ are fulfilled. Hence, from \eqref{eq: g_u,x finite} we get the existence of $\mu >0 $ and  neighborhood $U$ of $\bar{x}$ on which 

\begin{equation}\label{eq: g_l,x finite 2}
\forall \; x\in U: -\infty <  g_{l}(x,\bar{y}) \leq  \Psi_e(\mu e-\bar{y}).
\end{equation} Taking the supremum over $\bar{y} \in F(\bar{x})$ in \eqref{eq: g_l,x finite 2}, we get

\begin{equation}\label{eq: f_l,x finite?}
\forall \; x\in U: -\infty < f_{l,\bar{x}}(x) \leq \sup_{\bar{y}\in F(\bar{x})} \Psi_e(\mu e-\bar{y}).
\end{equation} Now, since $F$ is locally $l$-lower bounded at $\bar{x},$ in particular $F(\bar{x})\subseteq -\mu e +K.$ By the monotonicity of $\Psi_e,$ we now obtain $$ \sup_{\bar{y}\in F(\bar{x})} \Psi_e(\mu e-\bar{y})\leq \Psi_e(2\mu e)=2\mu.$$ This, together with \eqref{eq: f_l,x finite?}, implies that $f_{l,\bar{x}}$ is finite and upper bounded on $U.$ The statement follows.\\

$(ii)$ Let us now prove that $f_{u,\bar{x}}$ is convex. Indeed, take any $x_1,x_2\in X$ and $\lambda \in (0,1),$ Again, by denoting $x_\lambda = \lambda x_1+(1-\lambda) x_2$ and $F_\lambda= \lambda F(x_1)+ (1-\lambda) F(x_2),$ we have
\begin{eqnarray*}
f_{u,\bar{x}}(x_\lambda) & = & \sup_{y\in F(x_\lambda)} g_{u,\bar{x}}(y)\\
                         &\overset{(\;\mathrm{by }\;  \mathrm{convexity } \; \mathrm{of}\;  F\;)}{\leq}&  \sup_{y \in F_\lambda} g_{u,\bar{x}}(y)\\
                         & = &  \sup_{(y_1,y_2) \in F(x_1)\times F(x_2)} g_{u,\bar{x}}(\lambda y_1+(1-\lambda)y_2 )\\
                         & \overset{(\;\mathrm{by }\;  \mathrm{ convexity} \; \mathrm{of}\; g_{u,\bar{x}}\;)}{\leq} & \sup_{(y_1,y_2) \in F(x_1)\times F(x_2)}   \lambda g_{u,\bar{x}}(y_1) +(1-\lambda ) g_{u,\bar{x}}(y_2)\\
                         & = & \lambda f_{u,\bar{x}}(x_1) +(1-\lambda) f_{u,\bar{x}}(x_2),
\end{eqnarray*} as desired.

Now, assume that $F$ is locally $u$-upper bounded at $\bar{x}$ and let $U$ be the neighborhood on which the boundedness property holds. Again, in order to prove the second part it suffices to show that $f_{u,\bar{x}}$ is finite and upper bounded on a neighborhood of $\bar{x}.$ We proceed as follows: since $\bar{x}\in \Int \operatorname{dom} F,$ we can assume without loss of generality that $U \subseteq \Int\operatorname{dom } F.$ Moreover, since in particular the assumptions of Lemma \ref{lem: convexity inner functions} $(ii)$ are fulfilled, we get that $g_{u,\bar{x}}(y)> -\infty$ for every $y\in Y.$ Taking any selection $\theta$ of $F$ on $U,$ we deduce that 

$$\forall \; x\in U: \; -\infty < g_{u,\bar{x}}(\theta(x))\leq f_{u,\bar{x}}(x).$$ On the other hand, recall that from Lemma \ref{lem: convexity inner functions} $(ii)$ the functional $g_{u,\bar{x}}$ is $K$-monotone and finite. Taking this into account and the fact that $F(x)-K \subseteq \mu e -K $ for every $x\in U,$ we obtain 

$$\forall \; x\in U:\;  f_{u,\bar{x}}(x)\leq \sup_{y \in \mu e-K} g_{u,\bar{x}}(y) = g_{u,\bar{x}}(\mu e) < + \infty.$$ The theorem is proved.
\end{proof}

Next, we prove that the Lipschitzianity of the set-valued mapping is also transfered to the corresponding scalarization functionals. The following proposition is crucial. 

\begin{Proposition}\label{prop: lip marginal}
Let Assumption \ref{ass 1} be fulfilled and let $f:X\times Y \to \overline{\R}$ be a given functional. Consider the associated marginal functions $\varphi, \Phi:X\to \overline{\R}$ defined as

\begin{equation*}
\varphi(x):= \inf_{y\in F(x)} f(x,y), \;\;\; \Phi(x):= \sup_{y\in F(x)} f(x,y).
\end{equation*} Suppose that $F$ is Lipschitzian on a set $U\subseteq X$ with constant $\ell>0$ and that  $f$ is Lipschitzian on the set $(U \times Y) \cap \gph F$ with constant $\ell'>0.$ The following statements are true:

\begin{enumerate}

\item If $\varphi(\bar{x}) > -\infty$ for some $\bar{x} \in U,$ then $\varphi$ is Lipschitzian on $U$ with constant $\ell'(1+\ell).$ 

\item If $\Phi(\bar{x}) < +\infty$ for some $\bar{x} \in U,$ then $\Phi$ is Lipschitzian on $U$ with constant $\ell'(1+\ell).$ 

\end{enumerate}  

\begin{proof} We only prove $(i),$ since the proof of $(ii)$ is very similar. Take $x, x' \in U$ and let $\ell, \ell'>0$ be the Lipschitzian constants of $F$ and $f$ respectively. Then, because $F$ is Lipschitzian on $U,$ 

\begin{equation*}
\forall\; y' \in F(x'), \exists \; y\in F(x): \|y-y'\|\leq \ell\|x-x'\|.
\end{equation*}
Taking this into account, together with the Lipschitz continuity of $f$ on $(U\times Y)\cap \gph F,$ we have 

\begin{eqnarray*}
\forall\; y' \in F(x'), \exists \; y\in F(x): f(x,y) & \leq & f(x',y') + \ell'(\|x-x'\| + \|y-y'\|)\\
       & \leq & f(x',y') + \ell'(1+\ell) \|x-x'\|.
\end{eqnarray*} This implies 

\begin{equation}\label{eq: lipsch marginal funct}
\varphi(x) \leq \varphi(x') + \tilde{\ell} \|x-x''\|,
\end{equation} with $\tilde{\ell}:= \ell'(1+\ell).$ Since $\varphi(\bar{x})> -\infty,$ we can substitute $x = \bar{x}$ in \eqref{eq: lipsch marginal funct} to obtain that $\varphi(x') > -\infty $ for every $x' \in U.$ From this, it follows that  $\varphi$ is Lipschitzian on $U.$

\end{proof}

\end{Proposition}

Next lemma is an immediate consequence of Proposition \ref{prop: lip marginal}, Proposition \ref{prop: properties of tammer function} $(ii)$ and Proposition \ref{prop: g =0}. 

\begin{Lemma}\label{lem: Lipschitz of inner functions}
Let Assumption \ref{ass 1} be fulfilled. Consider the lower and upper inner functions given in Definition \ref{def- scalarization} and let $\rho$ be the Lipschitz constant of $\Psi_e$.  The following statements hold: 
\begin{enumerate}

\item If $F$ is Lipschitzian with constant $\ell>0$ on a neighborhood $U$ of $\bar{x}$ and there exists $\bar{y} \in Y$ with $g_{l}(\bar{x},\bar{y})> - \infty,$ then $g_{l}$ is Lipschitzian on $U\times Y$ with constant $\rho(1+\ell).$ In particular, the condition $g_{l}(\bar{x},\bar{y})> - \infty$ can be replaced by $\bar{y} \in \WMin(F(\bar{x}),K).$

\item The functional $g_{u,\bar{x}}$ is Lipschitzian on $Y$ with constant $\rho$ if and only if $g_{u,\bar{x}}(\bar{y})>-\infty$ for some $\bar{y} \in Y.$ \\ In particular, this is true if $\WMax(F(\bar{x}),K)\neq \emptyset.$  
\end{enumerate}

\end{Lemma}
\begin{proof}
$(i)$ Consider the set-valued mapping $\tilde{F}: X\times Y \rightrightarrows Y$ and the functional $\tilde{f}: X\times Y\times Y \to \R$ defined as

$$\tilde{F}(x,y):= F(x),\;\; \tilde{f}(x,y,z):= \Psi_e(z-y).$$ Apply now Proposition \ref{prop: lip marginal} $(i)$ with $\varphi := g_l, \; F := \tilde{F}$ and $f:= \tilde{f}$ to obtain the Lipschitzianity of $g_l.$ If $\bar{y} \in \WMin(F(\bar{x}),K),$ then it follows from Proposition \ref{prop: g =0} $(i)$ that $g_l(\bar{x},\bar{y}) = 0 > -\infty.$

$(ii)$ Follows easily from the fact that $g_{u,\bar{x}}$ is the finite infimum of a fixed family of Lipschitzian functionals on $Y.$ \\
Of course, when $\bar{y} \in \WMax(F(\bar{x}),K)\neq \emptyset,$ we get $g_{u,\bar{x}}(\bar{y}) = 0 > -\infty$ from Proposition \ref{prop: g =0} $(ii).$
\end{proof}

We can now establish the Lipschitzianity of the scalarizing functionals $f_{l,\bar{x}}$ and $f_{u,\bar{x}}.$

\begin{Theorem}\label{thm: Lipschitz of scalarizations}
Let Assumption \ref{ass 1} be fulfilled. For $r\in \{l,u\},$ consider the functional $f_{r,\bar{x}}$ given by \eqref{eq f_r} and suppose that $F$ is locally Lipschitzian at $\bar{x}.$ The following statements hold: 
\begin{enumerate}
\item If $\WMin(F(\bar{x}),K)\neq \emptyset,$  then $f_{l,\bar{x}}$ is locally Lipschitzian at $\bar{x}.$

\item If $\WMax(F(\bar{x}),K)\neq \emptyset,$  then $f_{u,\bar{x}}$ is locally Lipschitzian at $\bar{x}.$
\end{enumerate}
\end{Theorem}
\begin{proof} 
$(i)$ Consider the constant set-valued mapping $\tilde{F}: X \rightrightarrows Y$ given by $\tilde{F}(x):= F(\bar{x}) $ for every $x \in X.$ By Lemma \ref{lem: Lipschitz of inner functions} $(i),$ we know that $g_l$ is Lipschitzian on $U\times Y,$ where $U$ is a neighborhood of $\bar{x}$ on which $F$ is Lipschitzian. Furthermore, according to Proposition \ref{prop: g =0} $(iii),$ we have $f_{l,\bar{x}}(\bar{x})=0 < + \infty.$ Hence, the Lipschitzianity of $f_{l,\bar{x}}$ around $\bar{x}$ follows from Proposition \ref{prop: lip marginal} $(ii)$ with $\varphi := f_{l,\bar{x}}, F := \tilde{F}$ and $f:= g_l.$ 

$(ii)$ Similarly, consider the functional  $\tilde{f}: X\times Y \rightarrow \R$ given by $\tilde{f}(x,y):= g_{u,\bar{x}}(y) $ for every $(x,y) \in X\times Y.$ From \ref{lem: Lipschitz of inner functions} $(ii),$ we get that $\tilde{f}$ is Lipschitzian on $X\times Y.$ In addition, Proposition \ref{prop: g =0} $(iii)$ tells us that $f_{u,\bar{x}}(\bar{x})=0 < + \infty.$ Hence, the Lipschitzianity of $f_{u,\bar{x}}$ around $\bar{x}$ follows from Proposition \ref{prop: lip marginal} $(ii)$ with $\varphi := f_{u,\bar{x}}, F := F$ and $f:= \tilde{f}.$

\end{proof}

\section{Subdifferential of the scalarizing functional associated to the lower less relation}\label{sec: subdifferential l}

In this part, we derive upper estimates for Mordukhovich's subdifferential of the scalarizing functional $f_{l,\bar{x}}$ studied in Section \ref{sec: scalarizing function}. Our upper estimates are given in terms of the coderivative of the set-valued objective map $F$ and are based in Theorem \ref{thm: convex inf functions subdif} and Theorem \ref{thm: basic subdif marginal functions theorem}. These motivates the definition  of the following solution maps.

\begin{Definition}\label{def: solution maps lower}

Let Assumption 1 be fulfilled. 

\begin{enumerate}
\item The lower inner solution map $S^{l,1}_F: X\times Y \rightrightarrows Y$ is defined as

$$S^{l,1}_F(x,y):= \{z\in F(x): \Psi_e(z-y)= g_l(x,y)\}.$$

\item The  lower outer solution map $S^{l,2}_F: X \rightrightarrows Y$ is defined as

$$S^{l,2}_F(x):= \{y\in F(\bar{x}): f_{l,\bar{x}}(x)= g_{l}(x,y)\}.$$ 

\end{enumerate}

\end{Definition}

\begin{Remark}
According to Remark \ref{rem: replace F by epiF}, the functionals $g_l$ and $f_{l,\bar{x}}$ are invariant under replacement of $F$ by $\mathcal{E}_F.$ However, although the set-valued mappings $S^{l,i}_F$ and  $S^{l,i}_{\mathcal{E}_F}$ are based on the same functionals ($i=1,2$), we always have $S^{l,i}_F(\cdot) \subseteq  S^{l,i}_{\mathcal{E}_F}(\cdot)$ and the inclusions can be strict.

\end{Remark}

 We divide the analysis in two cases, corresponding to whether $F$ is $\preceq^{(l)}_K$-convex or locally Lipschitzian at $\bar{x}.$ We start the study with the convex case. The next lemma shows an exact  formula for the subdifferential of the inner function  given in Definition \ref{def- scalarization} $(i).$ It is worth mentioning that a similar version of this result was recently obtained in \cite[Lemma 2]{ha2019}, but assuming the separability of $X.$ 

\begin{Lemma}\label{lem: subdif inner g lower}

Let Assumption \ref{ass 1} be fulfilled and, for $\bar{y} \in \WMin(\mathcal{E}_F(\bar{x}),K),$  consider the functional $g_{l,\bar{y}}:=g_{l}(\cdot,\bar{y})$. Assume in addition that $F$ is $\preceq^{(l)}_K$-convex and locally $l$-bounded at $\bar{x}.$  Then, 
 
\begin{equation}\label{eq: subdif inner g lower}
\partial g_{l,\bar{y}}(\bar{x})= D^* \mathcal{E}_F(\bar{x},\bar{y})\left[\partial \Psi_e(0)\right].
\end{equation}

%\label{eq: subdiff inner g lower coderivative}

\end{Lemma}
\begin{proof} The result will be a simple consequence of Theorem \ref{thm: convex inf functions subdif}. Indeed, note that according to Remark \ref{rem: replace F by epiF} we can write

$$g_{l,\bar{y}}(x)= \inf_{y\in \mathcal{E}_F(x)} f(x,y),$$ where $f: X\times Y \to \R$ is defined as $f(x,y) = \Psi_e(y-\bar{y}).$  Since $F$ is $\preceq^{(l)}_K$- convex, we have that $\mathcal{E}_F$ is a convex set-valued mapping. It is also obvious that $f$ is proper and convex. Moreover, by Proposition \ref{prop: g =0} $(i)$, we have that $g_{l,\bar{y}}(\bar{x})=0\neq -\infty.$ According to Proposition \ref{prop: properties of tammer function} $(ii),$ $f$ is Lipschitzian on $X\times Y$ and hence the regularity condition $(ii)$ in Theorem \ref{thm: convex inf functions subdif} is satisfied. In this case, the solution map is just $S^{l,1}_{\mathcal{E}_F}(\cdot,\bar{y}).$ According to Proposition  \ref{prop: g =0} $(i)$  and Proposition \ref{prop: properties of tammer function} $(vi),$ we get  

\begin{equation}\label{eq: S^1l(xbar)}
S^{l,1}_{\mathcal{E}_F}(\bar{x},\bar{y})= \{y \in \mathcal{E}_F(\bar{x}) \mid \; \Psi_e(y-\bar{y}) = 0\}= \mathcal{E}_F(\bar{x}) \cap (\bar{y}- \bd K) .
\end{equation} Since $0 \in  \bd K,$ it follows that $\bar{y} \in S^{l,1}_{\mathcal{E}_F}(\bar{x},\bar{y}).$ Applying now Theorem \ref{thm: convex inf functions subdif}, we obtain 

\begin{eqnarray*}
\partial g_{l,\bar{y}}(\bar{x}) & =  &\bigcup_{(x^*,y^*) \in \partial f(\bar{x},\bar{y})}\left[ x^* +D^* \mathcal{E}_F (\bar{x},\bar{y})(y^*) \right]\\
                                & =  &\bigcup_{(x^*,y^*) \in \{0\} \times \partial \Psi_e(0)}\left[ x^* +D^* \mathcal{E}_F (\bar{x},\bar{y})(y^*) \right]\\
                                & =  &\bigcup_{y^* \in  \partial \Psi_e(0)}D^* \mathcal{E}_F (\bar{x},\bar{y})(y^*),
\end{eqnarray*} which proves the statement.

\end{proof}

\begin{Lemma}\label{lem: monotonicity of coderivative}

Let Assumption \ref{ass 1}  be fulfilled  and take points $(\bar{x},\bar{y}_1),(\bar{x},\bar{y}_2) \in \gph \mathcal{E}_F$ such that $\bar{y}_1\preceq_K \bar{y}_2.$ If  $F$ is $\preceq^{(l)}_K$- convex, then:  

 $$\forall \; y^* \in K^*: \; D^* \mathcal{E}_F(\bar{x},\bar{y}_2)(y^*) \subseteq D^* \mathcal{E}_F(\bar{x},\bar{y}_1)(y^*).$$
\end{Lemma}
\begin{proof}
%Since $F$ is $(K,l)$- convex, it follows that $\epi F = \gph \mathcal{E}_F$ is a convex set and hence $N_C(\gph \mathcal{E}_F,(\bar{x},\bar{y}_i))= N(\gph \mathcal{E}_F,(\bar{x},\bar{y}_i)),$ for $i=1,2.$  
Fix $y^* \in K^*$ and $x^*\in D^* \mathcal{E}_F(\bar{x},\bar{y}_2)(y^*).$ Since  $\bar{y}_1- \bar{y}_2 \in -K,$  we have that  $\langle y^*, \bar{y}_1- \bar{y}_2 \rangle \leq 0.$ Then, for every $(x,y) \in \gph \mathcal{E}_F,$ we have

\begin{eqnarray*}
\langle x^*, x- \bar{x} \rangle & \leq & \langle y^*, y- \bar{y}_2 \rangle\\
                                & = &  \langle y^*, y- \bar{y}_1 \rangle +\langle y^*, \bar{y}_1- \bar{y}_2 \rangle\\
                                & \leq & \langle y^*, y- \bar{y}_1 \rangle,
\end{eqnarray*} which implies that $(x^*, -y^*) \in N((\bar{x},\bar{y}_1),\gph \mathcal{E}_F).$ The statement is proved.
\end{proof}

The following concept was introduced in \cite{ha2001}.
\begin{Definition}

Let Assumption \ref{ass 1}  be fulfilled and consider $A \subseteq Y.$ We say that $A$ is strongly $K$- compact if there exists a compact set $B\subseteq A$ such that $B \in [A]^{(l)}.$
\end{Definition}

\begin{Theorem}\label{thm: subdiff lower l scalar}

Let Assumption \ref{ass 1} be satisfied. Suppose that $F$ is $\preceq^{(l)}_K$- convex and locally $l$-bounded at $\bar{x}.$ Furthermore, assume that $F(\bar{x})$ is strongly $K$-compact. Then, 

$$\partial f_{l,\bar{x}}(\bar{x})=\overline{\operatorname{conv}}^*\left(\bigcup_{\bar{y} \in \Min(F(\bar{x}),K)} D^* \mathcal{E}_F(\bar{x},\bar{y})\left[\partial \Psi_e(0)\right]\right).$$

\end{Theorem}
\begin{proof}
Under the assumptions of the theorem we can apply Theorem \ref{thm: convexity of scalarizations } $(i)$ to obtain that the functional $f_{l,\bar{x}}$ is convex and continuous at $\bar{x}.$ Hence, by \cite[Proposition 1.11]{Phe1989}, we have $\partial f_{l,\bar{x}}(\bar{x}) \neq \emptyset.$ Since $F(\bar{x})$ is strongly $K$- compact, there exists a compact set $A \subseteq F(\bar{x})$ such that $A+K=F(\bar{x})+K.$ Applying \cite[Lemma 4.7]{jahn2011}, we get that
\begin{equation}\label{eq: =min sets}
\Min(F(\bar{x}),K)= \Min(A,K)\neq \emptyset.
\end{equation} As in Lemma \ref{lem: subdif inner g lower} we consider, for $\bar{y} \in F(\bar{x}),$ the functional $g_{l,\bar{y}}:= g_l(\cdot,\bar{y}).$ Then, according to Proposition \ref{prop: g =0} $(i),$ the functional $g_l(x,\cdot)$ is $-K$ monotone for any $x \in X$. This implies

\begin{eqnarray*}
f_{l,\bar{x}}(x)&= &\sup_{\bar{y}\in F(\bar{x})} g_{l}(x,\bar{y})= \sup_{\bar{y}\in F(\bar{x})+K} g_{l}(x,\bar{y})= \sup_{\bar{y}\in A+K} g_{l}(x,\bar{y})\\
&= & \sup_{\bar{y}\in A} g_{l}(x,\bar{y}) = \sup_{\bar{y}\in A} g_{l,\bar{y}}(x).
\end{eqnarray*}

The above equation implies that $f_{l,\bar{x}}$ can be expressed as the pointwise supremum of the parametric family $\{g_{l,\bar{y}}\}_{\bar{y} \in A}$. In this context, it is stated in \cite[Proposition 4.5.2]{Schirotzek2007}  an exact formula for the subdifferential of the maximum of convex functions. In order to apply this proposition, it is sufficient to verify the following statements:

\begin{itemize}
\item $(A, \|\cdot\|)$ is a compact Hausdorff space.

This is obvious given our compactness assumption.

\item For any $ \bar{y} \in A,$ the functional $g_{l,\bar{y}}$ is convex and continuous at $\bar{x}.$

Since $A\subseteq F(\bar{x}),$ the statement follows directly from Lemma \ref{lem: convexity inner functions} $(i).$ 

\item For every $x\in X,$ the functional $g_l(x,\cdot)$  is u.s.c at every point of $A$.

Indeed, fix $x\in X$ and take $\bar{y}\in A,\alpha\in \Bbb R$ such that $g_l(x,\bar{y})<\alpha.$ This is equivalent to 

$$\inf_{y\in F(x)} \Psi_e(y-\bar{y}) < \alpha,$$ and hence we can find $y' \in F(x)$ such that $\Psi_e(y'-\bar{y})<\alpha.$ Because of the continuity of $\Psi_e,$ we can find a neighborhood $V(\bar{y})$ of $\bar{y}$ such that for every $z\in V(\bar{y}),$ the inequality $ \Psi_e(y'- z)< \alpha$ holds. This, together with the definition of $g_l(x, \cdot),$ gives us

$$\forall\; z\in V(\bar{y})\cap A: \; g_l(x,z)\leq \Psi_e(y'- z)< \alpha,$$ as desired. 

\end{itemize}

Applying now \cite[Proposition 4.5.2]{Schirotzek2007}, we obtain that 

\begin{equation}\label{eq: subdif l scal partial form}
\partial f_{l,\bar{x}}(\bar{x})=\overline{\operatorname{conv}}^*\left(\bigcup_{\bar{y} \in \tilde{S}} \partial g_{l,\bar{y}}(\bar{x}) \right),
\end{equation} where $$\tilde{S}= \{\bar{y}\in A : g_{l,\bar{y}}(\bar{x})= f_{l,\bar{x}}(\bar{x})\}.$$ Recall that $\WMin(F(\bar{x}),K) \neq \emptyset$ according to \eqref{eq: =min sets}. Then, by Proposition \ref{prop: g =0} $(iii),$ we know that $f_{l,\bar{x}}(\bar{x})=0.$ Hence, $\bar{y} \in \tilde{S}$ if and only if $g_{l,\bar{y}}(\bar{x})=0.$  Fix $\bar{y}\in A.$  Note that, because of the monotonicity of $\Psi_e,$ we have
\begin{eqnarray*}
g_{l,\bar{y}}(\bar{x}) &= & \inf_{y\in F(\bar{x})} \Psi_e(y-\bar{y})
                       = \inf_{y\in F(\bar{x})+K} \Psi_e(y-\bar{y})\\
                      & = &\inf_{y\in A+ K} \Psi_e(y-\bar{y})
                       = \inf_{y\in A} \Psi_e(y-\bar{y}).
\end{eqnarray*}
Then, following the same lines in the proof of Proposition \ref{prop: g =0} $(i),$ we get $$\inf \limits_{y\in A} \Psi_e(y-\bar{y})=0 \Longleftrightarrow \bar{y}\in \WMin(A,K).$$ This shows that 

\begin{equation}\label{eq: active set in convex case l}
\tilde{S}= \WMin(A,K).
\end{equation}
Now, since $A$ is compact, we can apply \cite[Proposition 9.3.7]{KTZ} to obtain that $A$ satisfies the so called domination property, i.e, 

\begin{equation}\label{eq: domination property}
A \subseteq \Min (A,K)+K.
\end{equation}
Hence, taking into account \eqref{eq: subdif l scal partial form}, \eqref{eq: active set in convex case l} and Lemma \ref{lem: subdif inner g lower},  we obtain

\begin{equation}\label{eq: partial equality subdif f_l}
	\partial f_{l,\bar{x}}(\bar{x})  = \overline{\operatorname{conv}}^*\left(\bigcup_{\bar{y} \in \WMin(A,K)} D^* \mathcal{E}_F(\bar{x},\bar{y})\left[\partial \Psi_e(0)\right]\right).
\end{equation}
By \eqref{eq: domination property}, for every $\bar{y} \in \WMin(A,K)$ there exists $\bar{y}_1 \in \Min(A,K)$ such that $ \bar{y}_1 \preceq_K \bar{y}.$ This, together with the fact that $\partial \Psi_e(0) \subseteq K^*,$ allows us to apply  Lemma \ref{lem: monotonicity of coderivative} to obtain 

\begin{equation}\label{eq: monotone cod}
	D^* \mathcal{E}_F(\bar{x},\bar{y})\left[\partial \Psi_e(0)\right] \subseteq  D^* \mathcal{E}_F(\bar{x},\bar{y}_1)\left[\partial \Psi_e(0)\right] .
\end{equation}

Combining equations \eqref{eq: partial equality subdif f_l} and  \eqref{eq: monotone cod}, we have 

\begin{eqnarray*}
\partial f_{l,\bar{x}}(\bar{x}) &= &\overline{\operatorname{conv}}^*\left(\bigcup_{\bar{y} \in \WMin(A,K)} D^* \mathcal{E}_F(\bar{x},\bar{y})\left[\partial \Psi_e(0)\right]\right) \\
&\subseteq & \overline{\operatorname{conv}}^*\left(\bigcup_{\bar{y}_1 \in \Min(A,K)} D^* \mathcal{E}_F(\bar{x},\bar{y}_1)\left[\partial \Psi_e(0)\right]\right) .        
\end{eqnarray*}
Since the reverse inclusion is obviously true, we obtain 

$$\partial f_{l,\bar{x}}(\bar{x}) =\overline{\operatorname{conv}}^*\left(\bigcup_{\bar{y} \in \Min(A,K)} D^* \mathcal{E}_F(\bar{x},\bar{y})\left[\partial \Psi_e(0)\right]\right).$$  The desired result follows from \eqref{eq: =min sets}.
\end{proof}

Next, we analyze the case on which $F$ is locally Lipschitzian at $\bar{x}.$ Similar to the convex case, we start by establishing an upper estimate of the subdifferential of the inner function.

\begin{Lemma}\label{lem: upper estimate inner lower}
Let Assumption \ref{ass 1} be fulfilled  with $X, Y$ being Asplund, and let $\bar{y} \in \WMin(F(\bar{x}),K).$ Suppose also that:

\begin{enumerate}

\item $F$ is closed at $\bar{x},$

\item $S^{l,1}_F(x,y)$ is inner semicompact at $(\bar{x},\bar{y})$,

\item $\gph F$ is locally closed around every point in the set $\{\bar{x}\}\times F(\bar{x})\cap(\bar{y}-  \bd K).$
\end{enumerate}

Then,

\begin{equation}\label{eq: subdif f1 prelim 2}
\partial g_l(\bar{x},\bar{y})\subseteq \bigcup_{\underset{z^* \in \partial \Psi_e(\bar{z}-\bar{y})}{\bar{z}\in F(\bar{x})\cap (\bar{y}-\bd K)}}  D^*F(\bar{x},\bar{z})(z^*) \times \{-z^*\}.
\end{equation} 
\end{Lemma}
\begin{proof}
Consider the set-valued mapping $\tilde{F}: X\times Y \rightrightarrows Y$ and the functional $f:X\times Y\times Y \to \R$ defined as

$$\tilde{F}(x,y):= F(x),\;\; f(x,y,z):= \Psi_e(z-y).$$ Thus, we have 

$$g_l(x,y)= \inf_{z\in \tilde{F}(x,y)}f(x,y,z).$$ Now, we check that it is possible to apply Theorem \ref{thm: basic subdif marginal functions theorem}. First, note that the associated solution map in this case is just $S^{l,1}_F,$ from Definition \ref{def: solution maps lower}. Next, observe that  $g_l(\bar{x},\bar{y})=0 $ by Proposition \ref{prop: g =0} $(i)$. Hence, using the representability property of $\Psi_e$ we get 

\begin{equation}\label{eq: Sl1(xbar,ybar)}
S^{l,1}_F(\bar{x},\bar{y})=\{z\in F(\bar{x}):\Psi_e(z-\bar{y})=0\}= F(\bar{x})\cap (\bar{y}-\bd K) \supseteq \{\bar{y}\} \neq  \emptyset.
\end{equation} We proceed to check that the hypothesis of the theorem are fulfilled.

\begin{itemize}

\item $\tilde{F}$ is  closed at $(\bar{x},\bar{y}).$

This is obvious given the definition of $\tilde{F}$ and condition $(i)$ above.

\item $S^{l,1}_F$  is inner semicompact at $(\bar{x}, \bar{y}).$

This is precisely condition $(ii)$ in the lemma.

\item There is a neighborhood $U'$ of $(\bar{x},\bar{y})$  such that $f$ is Lipschitzian on $U' \times Y.$

This follows directly from the definition of $f$ and Proposition \ref{prop: properties of tammer function} $(ii).$

\item $\gph \tilde{F}$ is locally closed around every point in the set $\{(\bar{x},\bar{y})\}\times S^{l,1}_F(\bar{x},\bar{y}).$

Taking into account \eqref{eq: Sl1(xbar,ybar)}, the statement follows from condition $(iii).$

\end{itemize}

Applying now Theorem \ref{thm: basic subdif marginal functions theorem} we obtain
\begin{equation}\label{eq: subdif f1 prelim}
\partial g_l(\bar{x},\bar{y})\subseteq  \bigcup_{\underset{(x^*,y^*,z^*) \in \partial f(\bar{y},\bar{y},\bar{z})}{\bar{z}\in F(\bar{x})\cap (\bar{y}-\bd K)}}  \bigg\{(x^*,y^*)+ D^* \tilde{F}(\bar{x},\bar{y},\bar{z})(z^*)\bigg\}.
\end{equation} 
We now simplify the above inclusion. The first step will be to examine $D^*\tilde{F}(\bar{x},\bar{y},\bar{z}).$  Note that 
$$\gph \tilde{F} = \{(x,y,z): z\in F(x)\}.$$ 
Hence, we obtain
$$N((\bar{x},\bar{y},\bar{z}),\gph \tilde{F})= \{(x^*,0,z^*) \in X^*\times Y^*\times Y^*: (x^*,z^*)\in N((\bar{x},\bar{z}),\gph F)\}.$$ From this we deduce that 
\begin{eqnarray}\label{eq: D^*F2}
 \nonumber D^*\tilde{F}(\bar{x},\bar{y},\bar{z})(z^*) & = & \{(x^*,0)\in X^*\times Y^*: (x^*,-z^*) \in N((\bar{x},\bar{z}),\gph F) \}\\ 
                                     & = & D^*F(\bar{x},\bar{z})(z^*) \times \{0\}.
\end{eqnarray}

Next, we compute $\partial f(\bar{x},\bar{y},\bar{z}).$  For this, we first note that $f$ is convex and continuous at every point. Considering the operator $T \in \mathcal{L}(X\times Y \times Y, Y)$ defined as $T(x,y,z):= z-y,$ we get $f= \Psi_e \circ T.$ By the classical chain rule in convex analysis  \cite[Proposition 3.28]{peypouquet2015}, we now obtain

$$\partial f(\bar{x},\bar{y},\bar{z})= \partial f(\bar{x},\bar{y},\bar{z}) = T^*[\partial \Psi_e(\bar{z}-\bar{y})]= T^*[\partial \Psi_e(\bar{z}-\bar{y})],$$ where $T^*$ denotes the adjoint operator of $T.$ Moreover, it is easy to check that $T^*(z^*)= (0,-z^*,z^*).$ Hence, we get
\begin{equation}\label{eq: subdif f2}
\partial f(\bar{x},\bar{y},\bar{z}) = \{0\}\times \bigcup_{z^* \in  \partial \Psi_e(\bar{z}-\bar{y}) }(-z^*,z^*).
\end{equation}
Substituting  now \eqref{eq: D^*F2} and \eqref{eq: subdif f2} into \eqref{eq: subdif f1 prelim}, the desired estimate is obtained.

\end{proof}

\begin{Theorem}\label{thm: upper estimate subdiff f_l}
In addition to Assumption 1, let $X$ and $Y$ be Asplund. Suppose also that:

\begin{enumerate}
\item $F$ is locally Lipschitzian at $\bar{x},$

\item $\WMin(F(\bar{x}),K)\neq \emptyset,$

\item $F$ is closed at $\bar{x}$,

\item $S^{l,1}_F$ is inner semicompact at every point of  $\{\bar{x}\}\times\WMin(F(\bar{x}),K),$

\item $S^{l,2}_F$ is inner semicompact at $\bar{x},$

\item $\gph F$ is locally closed around every point in the set $\{\bar{x}\}\times\WMin(F(\bar{x}),K).$

\end{enumerate} Then,

\begin{equation}\label{eq: upper estimate subdiff f_l}
\partial f_{l,\bar{x}}(\bar{x})\subseteq \overline{\operatorname{conv}}^*\left( \bigcup_{\bar{y}\in \WMin(F(\bar{x}),K)} \Bigg\{x^*\in X^*: \exists \;  y^* \in N(\bar{y},F(\bar{x})):  (x^*,y^*) \in G_{(\bar{x},\bar{y})} \Bigg\}\right),
\end{equation} where

$$G_{(\bar{x},\bar{y})}= \overline{\operatorname{conv}}^*\left( \bigcup_{\underset{z^* \in \partial \Psi_e(\bar{z}-\bar{y})}{\bar{z}\in F(\bar{x})\cap (\bar{y}-\bd K)}} D^*F(\bar{x},\bar{z})(z^*) \times \{-z^*\} \right).$$
\end{Theorem}
\begin{proof}

Consider the (constant) set-valued mapping $\tilde{F}:X\rightrightarrows Y$ defined as $\tilde{F}(x):=F(\bar{x})$ for every $x\in X.$ Then, we can write

$$f_{l,\bar{x}}(x)= \sup_{y\in \tilde{F}(x)}g_l(x,y).$$ Next, note that the solution map in this case is $S^{l,2}_F.$ Furthermore, as a consequence of $(ii)$ and Proposition \ref{prop: g =0} $(iii),$ we obtain $f_{l,\bar{x}}(\bar{x})=0.$ The definition of $\tilde{F}$ and $g_l$ allow us then to apply Proposition \ref{prop: g =0} $(i)$ to obtain that $$S^{l,2}_F(\bar{x})= \WMin(F(\bar{x}),K).$$ We now check that it is possible to apply Theorem \ref{thm: basic subdif marginal functions theorem} to obtain an upper estimate of Mordukhovich's subdifferential of $f_{l,\bar{x}}$ at $\bar{x}.$

\begin{itemize}

\item  $\tilde{F}$ is closed at $\bar{x}.$

It is easy to see that the closedness of $\tilde{F}$ at $\bar{x}$ is equivalent to the closedness of the set $F(\bar{x}).$ The statement follows from condition $(iii).$

\item $S^{l,2}_F$ is inner semicompact at $\bar{x}.$

This is precisely condition $(v).$

\item There is a neighborhood $U$ of $\bar{x}$ such that $g_l$ is Lipschitzian on $U\times Y.$

This follows from conditions $(i),(ii)$ and Lemma \ref{lem: Lipschitz of inner functions} $(i).$

\item $\gph \tilde{F}$ is locally closed around every point of the set  $\{\bar{x}\} \times S^{l,2}_F(\bar{x}).$

Again, this is deduced from the fact that $F(\bar{x})$ is a closed set, which is implied by $(iii).$
\end{itemize}

Hence, taking into account the Lipschitzianity of $f_{l,\bar{x}}$ from Theorem \ref{thm: Lipschitz of scalarizations} $(i)$, we obtain:

\begin{eqnarray}\label{eq: clark f_l preliminar}
\partial f_{l,\bar{x}}(\bar{x}) & = & \partial \left(-\inf_{y\in \tilde{F}(\cdot)} -g_l(\cdot,y) \right)(\bar{x})\nonumber\\
                                  & \overset{\left(\textrm{Remark }\ref{rem: subdif -f}\right)}{\subseteq} & - \overline{\operatorname{conv}}^* \left(  \partial \left(\inf_{y\in \tilde{F}(\cdot)} -g_l(\cdot,y) \right)(\bar{x})\right)\nonumber\\
                       &\overset{(\textit{Theorem }\ref{thm: basic subdif marginal functions theorem} )}{\subseteq} & -\overline{\operatorname{conv}}^*\left( \bigcup_{\underset{(x^*,y^*)\in \partial\left(- g_l\right)(\bar{x},\bar{y})}{\bar{y}\in S^{l,2}_F(\bar{x})}}  \bigg[ x^*+ D^* \tilde{F}(\bar{x},\bar{y})(y^*) \bigg] \right).
\end{eqnarray} Now, we examine $D^* \tilde{F}(\bar{x},\bar{y})$ for any $(\bar{x},\bar{y})\in X\times Y.$ Since $\gph \tilde{F}= X\times F(\bar{x}),$ we get in this case $N((\bar{x},\bar{y}),\gph \tilde{F})= \{0\} \times N(\bar{y}, F(\bar{x})).$ From this, we deduce that 

$$D^* \tilde{F}(\bar{x},\bar{y})(y^*)= \left\{
\begin{array}{ll}
      \{0\}, & \operatorname{if } y^* \in -N(\bar{y},F(\bar{x})),   \\
      \,\,\,\emptyset, & \operatorname{otherwise}. \\
\end{array} 
\right.  $$ Plugging this back into \eqref{eq: clark f_l preliminar} and taking into account that $S^{l,2}_F(\bar{x})= \WMin(F(\bar{x}),K),$ we obtain 
\begin{equation}\label{eq: subdif f-lower prelim}
\partial f_{l,\bar{x}}(\bar{x})\subseteq \overline{\operatorname{conv}}^*\left( \bigcup_{\bar{y}\in \WMin(F(\bar{x}),K)} \bigg\{x^*\in X^*: \exists \;  y^* \in  N(\bar{y},F(\bar{x})): -(x^*,y^*) \in \partial (-g_l)(\bar{x},\bar{y})\bigg\}\right). 
\end{equation} 

On the other hand, taking into account the Lipschitzianity of $g_l$ from  Lemma \ref{lem: Lipschitz of inner functions} $(i),$ for every $\bar{y} \in \WMin(F(\bar{x}),K)$ we also have:

\begin{eqnarray}\label{eq: estimate subdif -g_l}
\partial (-g_l)(\bar{x},\bar{y}) & \overset{\left(\textrm{Remark }\ref{rem: subdif -f}\right)}{\subseteq} & - \overline{\operatorname{conv}}^* \left(\partial g_l(\bar{x},\bar{y})\right)\nonumber\\
                                   & \overset{\left(\textrm{Lemma }\ref{lem: subdif inner g lower}\right)}{\subseteq} & -\overline{\operatorname{conv}}^* \left( \bigcup_{\underset{z^* \in \partial \Psi_e(\bar{z}-\bar{y})}{\bar{z}\in F(\bar{x})\cap (\bar{y}-\bd K)}}  D^*F(\bar{x},\bar{z})(z^*) \times \{-z^*\} \right).
\end{eqnarray}

Finally, by putting \eqref{eq: estimate subdif -g_l}
 back into \eqref{eq: subdif f-lower prelim}, we obtain our desired estimate.
\end{proof}

\begin{Remark}\label{rem: f_l invariant}
According to Remark \ref{rem: replace F by epiF}, the scalarizing functional $f_{l,\bar{x}}$ would remain unchanged if we substitute $F$ by a set-valued mapping $\tilde{F}: X \rightrightarrows Y$ of the form $\tilde{F}(x) = F(x) + A,$ with $A \subseteq K$ and $0 \in A.$ Hence, in Theorem \ref{thm: upper estimate subdiff f_l} we can substitute $F$ by any other set-valued mapping  $\tilde{F}$ of the above form. By doing this, we can obtain different (maybe sharper) upper estimates of $\partial f_{l,\bar{x}} (\bar{x}).$ This is worth keeping in mind when obtaining optimality conditions for set optimization problems, as these are based on the subdifferential of $f_{l,\bar{x}} (\bar{x}),$ see Section \ref{sec: opt cond}.
\end{Remark}

\begin{Remark}\label{rem: convexity of upper estimate is necessary}
Note that, since the upper estimate of $\partial f_{l,\bar{x}}(\bar{x})$ obtained in \eqref{eq: upper estimate subdiff f_l} is convex, it also constitutes an upper estimate of $\partial^\circ f_{l,\bar{x}}(\bar{x})$  according to \cite[Theorem 3.57]{Mordukhovich1}. However, as we will see in  Example \ref{ex: master example}, when applying this result to optimality conditions for set optimization problems, the convexity of the upper estimate  can not be removed very easily.

\end{Remark}

The following corollary shows that if $Y$ is finite dimensional our assumptions in Theorem \ref{thm: upper estimate subdiff f_l} are natural.
\begin{Corollary}\label{cor: f_l subdif}
Let Assumption \ref{ass 1} be fulfilled with $X$ being Asplund. Suppose that $Y$ is finite dimensional and that $\gph F$ is closed. Furthermore, assume that $F$ is locally Lipschitzian and locally bounded at $\bar{x}.$ Then, inclusion \eqref{eq: upper estimate subdiff f_l} holds. 
\end{Corollary}

\begin{proof}
Since $\gph F$ is closed, in particular we have that $F$ is closed valued. This, together with the local boundedness at $\bar{x}$ and the finite dimensionality of $Y,$ gives us the compactness of $F(\bar{x}).$ Hence, according to \cite[Theorem 6.3]{jahn2011}, we have $\WMin(F(\bar{x}),K)\neq \emptyset.$ Furthermore, the local boundedness of $F$ at $\bar{x}$ also implies that of the set-valued mappings $S^{l,1}_F$ and $S^{l,2}_F$ in the statement of Theorem \ref{thm: upper estimate subdiff f_l}. This, together with the fact that  $Y$ is finite dimensional gives us the inner semicompactness of $S^{l,1}_F$ and $S^{l,2}_F.$ Thus, all the conditions of Theorem  \ref{thm: upper estimate subdiff f_l} are satisfied. The statement follows.

\end{proof}

\section{Subdifferential of the scalarizing functional associated to the upper less relation}\label{sec: subdifferential u}

In this section, we compute an approximation of the subdifferential of the functional $f_{u,\bar{x}}$ given in Definition \ref{def- scalarization} at the point $\bar{x}.$ We start again by defining two useful solution maps. 

\begin{Definition}
Let Assumption \ref{ass 1}  be fulfilled. 

\begin{enumerate}
\item The upper inner solution map $S^{u,1}_F: Y \rightrightarrows Y$ is defined as

$$S^{u,1}_F(y) := \{z\in F(\bar{x}): g_{u,\bar{x}}(y)= \Psi_e(y-z)\}.$$

\item  The upper outer solution map $S^{u,2}_F: X \rightrightarrows Y$ is defined as

$$S^{u,2}_F(x):= \{y\in F(x): f_{u,\bar{x}}(x)=g_{u,\bar{x}}(y)\}.$$

\end{enumerate}

\end{Definition}

In the next lemma, we obtain upper estimates for the subdifferentials of the inner function in both the convex and Lipschitzian cases.
%The first step is to compute the subdifferential of $g_{u,\bar{x}}$ at a point $\bar{y} \in \WMax (\mathcal{H}_F(\bar{x}),K).$

\begin{Lemma}\label{lem: subdif inner g upper}

Let Assumption \ref{ass 1}  be fulfilled.  The following statements hold:

\begin{enumerate}
\item Let $\bar{y} \in \WMax (\mathcal{H}_F(\bar{x}),K) $ and suppose that $\mathcal{H}_F(\bar{x})$ is a convex and $K$-upper bounded set. Then, $g_{u,\bar{x}}$ is convex, continuous at $\bar{x}$ and

\begin{equation}\label{eq: upper estimate inner upper convex}
\partial g_{u,\bar{x}}(\bar{y})=  \partial \Psi_e(0)\cap N(\bar{y},\mathcal{H}_F(\bar{x})).
\end{equation}

\item Let $X$ and $Y$ be Asplund and fix $\bar{y} \in \WMax (F(\bar{x}),K).$ Suppose that:
 \begin{enumerate}
 \item  $F(\bar{x})$ is closed, 

\item  $S^{u,1}_F$ is inner semicompact at $\bar{y}.$ 
 \end{enumerate}Then,

\begin{equation}\label{eq: upper estimate inner upper}
\partial g_{u,\bar{x}}(\bar{y})\subseteq  \bigcup_{\bar{z} \in F(\bar{x})\cap (\bar{y}+ \bd K)} \partial \Psi_e(\bar{y}-\bar{z})\cap N(\bar{z},F(\bar{x})).
\end{equation}

\end{enumerate}

\end{Lemma}

\begin{proof}  Our statements will follow from Theorem \ref{thm: convex inf functions subdif} and Theorem \ref{thm: basic subdif marginal functions theorem} respectively. In order to see this, we consider $T \in \mathcal{L}(Y\times Y, Y)$ and $f:Y \times Y \to \R$ defined  respectively as 

$$T(y,z):= y-z, \; f(y,z):= (\Psi_e \circ T)(y,z).$$ Furthermore, we define the set-valued maps $\tilde{F}, \hat{F}: Y \rightrightarrows Y$  respectively as $\tilde{F}(y)= \mathcal{H}_F(\bar{x})$ and $\hat{F}(y) = F(\bar{x})$ for every $y \in Y.$ We can then write 

$$g_{u,\bar{x}}(y)=\inf_{z \in \tilde{F}(y)} f(y,z),$$ with corresponding solution map $S^{u,1}_{\mathcal{H}_F},$ and 

$$g_{u,\bar{x}}(y)=\inf_{z \in \hat{F}(y)} f(y,z), $$ with corresponding solution map $S^{u,1}_{F}.$ By    Proposition \ref{prop: g =0} $(ii)$ and Proposition \ref{prop: properties of tammer function} $(vi),$ we get 

\begin{eqnarray*}
S^{u,1}_{\mathcal{H}_F}(\bar{y}) & = & \{z \in \mathcal{H}_F(\bar{x}) \mid \; \Psi_e(\bar{y}-z) = g_{u,\bar{x}}(\bar{y})\} \\   
           & = & \{z \in \mathcal{H}_F(\bar{x}) \mid \; \Psi_e(\bar{y}-z) = 0\}    \\
           & = & \mathcal{H}_F(\bar{x})\cap (\bar{y}+ \bd K).\\                           
\end{eqnarray*} In particular, we deduce that $\bar{y}\in S^{u,1}_{\mathcal{H}_F}(\bar{x}).$ Similarly, we obtain 

\begin{equation}\label{eq: S^u,1(bary)}
S^{u,1}_F(\bar{y}) = F(\bar{x})\cap (\bar{y}+ \bd K).
\end{equation} On the other hand, it is obvious that $f$ is convex and continuous. Moreover, for any $\bar{z} \in Y,$ the chain rule of of convex analysis \cite[Proposition 3.28]{peypouquet2015} implies  

$$\partial f(\bar{y},\bar{z}) = T^* [\partial \Psi_e(T(\bar{y},\bar{z}))] = T^* [\partial \Psi_e(\bar{y}-\bar{z})],$$ where $T^* \in \mathcal{L}(Y^*, Y^*\times Y^*)$ is the adjoint operator of $T.$ It is easy to verify that in this case $T^* (y^*) = (y^*, -y^*).$ Hence, we get
 
 \begin{equation}\label{eq: subdiff composition tammer with linear}
 \partial f(\bar{y},\bar{z}) = \bigcup_{y^* \in \partial \Psi_e(\bar{y}-\bar{z})} (y^*, -y^*).
 \end{equation} We proceed now to analyze each case separately.

$(i)$ The convexity and continuity follows from Lemma \ref{lem: convexity inner functions} $(ii).$ The subdifferential formula will be a simple application of Theorem \ref{thm: convex inf functions subdif} and to do so, we check that the hypothesis are fulfilled. Indeed, by assumption, $\mathcal{H}_F(\bar{x})$ is a convex set and hence $\tilde{F}$ is a convex set-valued mapping. Moreover, from  Proposition \ref{prop: properties of tammer function} $(i), (ii)$ it follows that $f$ is a proper convex function that is continuous at any point of $\gph \tilde{F}$ and hence, in particular, the regularity condition $(ii)$ in Theorem \ref{thm: convex inf functions subdif} is satisfied. As a consequence of Proposition \ref{prop: g =0} $(ii),$ we also have that $\bar{y} \in \dom g_{u,\bar{x}}$ and $\dom g_{u,\bar{x}}(\bar{y})=0 < +\infty.$ 

Since  $\bar{y} \in S^{u,1}_{\mathcal{H}_F}(\bar{x}),$ we can apply now Theorem \ref{thm: convex inf functions subdif} to obtain 

\begin{equation}\label{eq: subdif g_u convex partial}
\partial g_{u,\bar{x}}(\bar{y}) = \bigcup_{(y^*,z^*) \in \partial f(\bar{y},\bar{y})}\bigg[ y^* + D^* \tilde{F}(\bar{y},\bar{y})(z^*) \bigg].
\end{equation} Next, we examine the term $D^* \tilde{F}(\bar{y},\bar{y})(z^*)$ in the above formula. Note that $\gph \tilde{F} = Y \times \mathcal{H}_F(\bar{x}).$ Hence, we get $N((\bar{y},\bar{y}),\gph \tilde{F})= \{0\}\times N(\bar{y},\mathcal{H}_F(\bar{x}))$ and from this it follows that, for any $y^* \in Y^*:$ 
 
\begin{eqnarray*}
D^*\tilde{F}(\bar{y},\bar{y})(-y^*) & = & \{z^* \in Y^* \mid \; (z^*,y^*) \in N(\bar{y},\mathcal{H}_F(\bar{x}))\} \\
                              & = & \{z^* \in Y^* \mid \; (z^*,y^*) \in \{0\}\times N(\bar{y},\mathcal{H}_F(\bar{x}))\}\\
                              & = &  \left\{
\begin{array}{ll}
      \{0\}, &  \textrm{ if } y^* \in N(\bar{y},\mathcal{H}_F(\bar{x})), \\
      \emptyset,  & \textrm{ otherwise} . \\

\end{array} 
\right. 
\end{eqnarray*}

Taking this into account together with \eqref{eq: subdiff composition tammer with linear}, we obtain the following in \eqref{eq: subdif g_u convex partial}:

\begin{eqnarray*}
\partial g_{u,\bar{x}}(\bar{y}) & = & \bigcup_{y^* \in \partial \Psi_e(0)}\bigg[ y^* + D^* \tilde{F}(\bar{y},\bar{y})(-y^*) \bigg]\\
                                 & = & \bigcup_{y^* \in \partial \Psi_e(0)} \left[y^* + \left\{
\begin{array}{ll}
      \{0\}, &  \textrm{ if } y^* \in N(\bar{y},\mathcal{H}_F(\bar{x})), \\
      \emptyset,  & \textrm{ otherwise} \\

\end{array} 
\right. \right] \\
                                  & = & \partial \Psi_e(0)\cap N(\bar{y},\mathcal{H}_F(\bar{x})),                     
\end{eqnarray*} as expected.

$(ii)$ In this case, we will apply Theorem \ref{thm: basic subdif marginal functions theorem} to obtain an upper estimate of $\partial g_{u,\bar{x}}(\bar{y}).$ We check that all the conditions of the theorem are fulfilled:  
\begin{itemize}

\item $\hat{F}$ is closed at $\bar{y}.$

This follows from condition $(a).$

\item $S^{u,1}_F$ is inner semicompact at $\bar{y}.$

This is just condition $(b).$

\item There exists a neighborhood $V$ of $\bar{y}$ such that $f$ is Lipschitzian on $V\times Y.$

 Follows directly from the Lipschitzianity of $\Psi_e$ in Proposition \ref{prop: properties of tammer function} $(ii).$

\item $\gph \tilde{F}$ is locally closed around every point in the set $\{\bar{y}\} \times S^{u,1}_F(\bar{y}).$

This is a consequence of $(a).$

\end{itemize}

Theorem \ref{thm: basic subdif marginal functions theorem} together with \eqref{eq: S^u,1(bary)} gives us now

\begin{equation}\label{eq: partial upper inner lip}
\partial g_{u,\bar{x}}(\bar{y})\subseteq  \bigcup_{\underset{ (y^*,z^*)\in \partial f(\bar{y},\bar{z}) }{\bar{z} \in F(\bar{x})\cap (\bar{y}+ \bd K)}}  \bigg[ y^* + D^*\hat{F}(\bar{y},\bar{z})(z^*)\bigg].
\end{equation} Analogous to the proof of statement $(i),$ we obtain

$$D^*\hat{F}(\bar{y},\bar{z})(z^*)= \left\{
\begin{array}{ll}
      \{0\}, & \operatorname{if } z^* \in -N(\bar{z}, F(\bar{x})),   \\
      \emptyset, & \operatorname{otherwise}. \\
\end{array} 
\right. $$  Finally, by substituting this and \eqref{eq: subdiff composition tammer with linear} into \eqref{eq: partial upper inner lip}, the desired estimate is obtained.
\end{proof}

Next, we state the main result of the section.

\begin{Theorem}\label{thm: subdif f_u lip case}
In addition to Assumption \ref{ass 1}, let $X$ and $Y$ be Asplund. Suppose also that:

\begin{enumerate}
\item $F$ is locally Lipschitzian at $\bar{x},$

\item $\WMax(F(\bar{x}),K)\neq \emptyset,$

\item $F$ is closed at $\bar{x},$

\item  $S^{u,1}_F$ is inner semicompact at every point in the set $\WMax(F(\bar{x}),K),$

\item $S^{u,2}_F$  is inner semicompact at $\bar{x}.$

\item $\gph F$ is locally closed around any point in  the set $\{\bar{x}\} \times \WMax (F(\bar{x}),K).$

\end{enumerate} 
Then,
%{\fontsize{9.5}{4}\selectfont
\begin{equation}\label{eq: upper estimate subdiff f_u}
\partial f_{u,\bar{x}}(\bar{x})\subseteq - \overline{\operatorname{conv}}^*\left( \bigcup_{\bar{y}\in \WMax(F(\bar{x}),K)}  D^*F(\bar{x},\bar{y})\left[  H_{(\bar{x},\bar{y})} \right] \right),
\end{equation} where 

$$H_{(\bar{x},\bar{y})}: =  -\overline{\operatorname{conv}}^*\left( \bigcup_{\bar{z} \in F(\bar{x})\cap (\bar{y}+ \bd K)} \partial \Psi_e(\bar{y}-\bar{z}) \cap N \left(\bar{z},F(\bar{x})\right)\right).$$
%}
\end{Theorem}

\begin{proof}
Consider the function $f:X\times Y \to Y$ defined as $f(x,y)=g_{u,\bar{x}}(y).$ By definition,  we have 

$$f_{u,\bar{x}}(x)=\sup_{y\in F(x)} f(x,y).$$ We verify that we can apply Theorem \ref{thm: basic subdif marginal functions theorem}. First, note that the solution map in this case is just $S^{u,2}_F.$ Hence, Proposition \ref{prop: g =0} $(iii)$ can be applied to obtain $f_{u,\bar{x}}(\bar{x})=0.$ Then, from Proposition \ref{prop: g =0} $(ii)$ we get 

\begin{equation}\label{eq: S^u,2_F(barx) }
S^{u,2}_F(\bar{x})= \WMax(F(\bar{x}),K)\neq \emptyset.
\end{equation} We proceed to check the rest of the assumptions:

\begin{itemize}

\item $F$ is closed at $\bar{x},$

This is  just condition $(iii)$ in the theorem.

\item $S^{u,2}_F$ is inner semicompact at $\bar{x},$

This is exactly condition $(v)$ in our theorem.

\item There is a neighborhood $U$ of $\bar{x}$ such that $f$ is Lipschitzian on $U\times Y.$

Follows directly from condition $(ii)$ and Lemma \ref{lem: Lipschitz of inner functions} $(ii).$

\item $\gph F$ is locally closed around every point in the set $\{\bar{x}\}\times S^{u,2}_F(\bar{x}).$

This follows from \eqref{eq: S^u,2_F(barx) } and condition $(vi)$ in the theorem.

\end{itemize}
Hence, taking into account the Lipschitzianity of $f_{u,\bar{x}}$ from Theorem \ref{thm: Lipschitz of scalarizations} $(ii)$, we obtain:

\begin{eqnarray}\label{eq: subdif u step1}
\partial f_{u,\bar{x}}(\bar{x}) & = & \partial \left(-\inf_{y\in F(\cdot)} -f(\cdot,y) \right)(\bar{x})\nonumber\\
                                  & \overset{\left(\textrm{Remark }\ref{rem: subdif -f}\right)}{\subseteq} & - \overline{\operatorname{conv}}^* \left(\partial \left(\inf_{y\in F(\cdot)} -f(\cdot,y) \right)(\bar{x})\right)\nonumber\\
                       &\overset{(\textrm{Theorem }\ref{thm: basic subdif marginal functions theorem}\; +\; \eqref{eq: S^u,2_F(barx) } )}{\subseteq} & -\overline{\operatorname{conv}}^*\left( \bigcup_{\underset{(x^*,y^*)\in \partial\left(- f\right)(\bar{x},\bar{y})}{\bar{y}\in \WMax(F(\bar{x}),K)}}  \bigg[ x^*+ D^* F(\bar{x},\bar{y})(y^*) \bigg] \right).
\end{eqnarray}  Note that $f$ is independent of the argument in the space $X.$ Furthermore, since $F$ is closed at $\bar{x},$ we also have that $F(\bar{x})$ is a closed set. Hence, together with condition $(iv),$ it is easy to see that the assumptions of Lemma \ref{lem: subdif inner g upper} are satisfied. Then, for any $\bar{y} \in \WMax(F(\bar{x}),K),$ we get:  

\begin{eqnarray}\label{eq: partial f enlarged}
\partial(- f)(\bar{x},\bar{y}) & = &\{0\} \times \partial (-g_{u,\bar{x}})(\bar{y}) \nonumber \\
                                 & \overset{\left(\textrm{Remark }\ref{rem: subdif -f}\right)}{\subseteq}  & -\{0\} \times  \overline{\operatorname{conv}}^*\left(\partial g_{u,\bar{x}}(\bar{y})\right)\nonumber\\
                                 & \overset{ (\textrm{Lemma \ref{lem: subdif inner g upper} $(ii)$})}{\subseteq} &   -\{0\} \times \overline{\operatorname{conv}}^*\left( \bigcup_{\bar{z} \in F(\bar{x})\cap (\bar{y}+ \bd K)} \partial \Psi_e(\bar{y}-\bar{z})\cap N(\bar{z},F(\bar{x}))\right).
\end{eqnarray} Substituting \eqref{eq: partial f enlarged} into \eqref{eq: subdif u step1},  we obtain the desired estimate.
\end{proof} 

\begin{Remark}\label{rem: f_u invariant}
Similar to Remark \ref{rem: f_l invariant}, the functional $f_{u,\bar{x}}$ remains unchanged if we substitute $F$ by $\tilde{F}: X \rightrightarrows Y$ of the form $\tilde{F}(x) = F(x) - A,$ with $A \subseteq K$ and $0 \in A.$  Hence, in Theorem \ref{thm: subdif f_u lip case} we can substitute $F$ by any other set-valued mapping  $\tilde{F}$ of the above form. From this, we can obtain different (maybe sharper) upper estimates of $\partial f_{u,\bar{x}} (\bar{x}),$ which can be translated into sharper optimality conditions set optimization problems, see Section \ref{sec: opt cond}.
\end{Remark}

\begin{Remark}
Similarly to Remark \ref{rem: convexity of upper estimate is necessary}, we mention that, although the upper estimate in \eqref{eq: upper estimate subdiff f_u} is convex (and hence we are also estimating $\partial^\circ f_{u,\bar{x}}(\bar{x})$), Example \ref{ex: master example} illustrates that  convexity is necessary.
\end{Remark}
The proof of the following corollary is similar to that of Corollary \ref{cor: f_l subdif}, and it is hence omitted.

\begin{Corollary}\label{cor: f_u subdif}
Let Assumption \ref{ass 1} be fulfilled with $X$ being Asplund. Suppose that $Y$ is finite dimensional and that $\gph F$ is closed. Furthermore, assume that $F$ is locally Lipschitzian and locally bounded at $\bar{x}.$ Then, inclusion \eqref{eq: upper estimate subdiff f_u} holds.
\end{Corollary}

We conclude this section with a sharper result in the convex case.

\begin{Theorem}\label{thm: subdif f_u convex case}
In addition to Assumption \ref{ass 1}, let $X$ and $Y$ be Asplund. Suppose  also that

\begin{enumerate}
\item $F$ is $\preceq^{(u)}_K$-convex and locally $u$-upper bounded at $\bar{x}$,

\item $\mathcal{H}_F$ is convex valued in a neighborhood of $\bar{x},$

\item $\mathcal{H}_F$ is closed at $\bar{x},$

\item $\WMax(\mathcal{H}_F(\bar{x}),K)\neq \emptyset,$

\item $S^{u,2}_{\mathcal{H}_F}(x)$  is inner semicompact at $\bar{x}.$

\item $\gph \mathcal{H}_F$ is locally closed around any point in  the set $\{\bar{x}\} \times \WMax (\mathcal{H}_F(\bar{x}),K).$

 \end{enumerate} 
Then,
%{\fontsize{9.5}{4}\selectfont
\begin{equation*}
\partial f_{u,\bar{x}}(\bar{x})\subseteq - \overline{\operatorname{conv}}^*\left( \bigcup_{\bar{y}\in \WMax(\mathcal{H}_F(\bar{x}),K)}  D^*\mathcal{H}_F(\bar{x},\bar{y})\left[-\partial \Psi_e(0) \cap N(\bar{y},\mathcal{H}_F(\bar{x}))   \right] \right).
\end{equation*}
%}

\end{Theorem}

\begin{proof} Because of conditions $(i)$ and $(ii),$ we can apply  \cite[Theorem 7.4.9]{TT2018} to obtain that $\mathcal{H}_F$ is locally Lipschitzian at $\bar{x}.$ Then, it is easy to see that assumptions $(i)-(iii), (v)-(vi)$ of Theorem \ref{thm: subdif f_u lip case} are satisfied if we replace $F$ by $\mathcal{H}_F.$ Since these assumptions are the only ones needed to obtain \eqref{eq: subdif u step1}, we can take into account Remark \ref{rem: f_u invariant} to get in this case 

\begin{equation}\label{eq: f_u convex aux}
\partial f_{u,\bar{x}}(\bar{x}) \subseteq -\overline{\operatorname{conv}}^*\left( \bigcup_{\underset{(x^*,y^*)\in \partial\left(- f\right)(\bar{x},\bar{y})}{\bar{y}\in \WMax(\mathcal{H}_F(\bar{x}),K)}}  \bigg[ x^*+ D^* \mathcal{H}_F(\bar{x},\bar{y})(y^*) \bigg] \right),
\end{equation} where $f$ is the same function defined in Theorem \ref{thm: subdif f_u lip case}. Similar to \eqref{eq: partial f enlarged}, but applying Lemma \ref{lem: subdif inner g upper} $(i)$ instead, we obtain

\begin{equation}\label{eq: aux f_u 2}
\partial(- f)(\bar{x},\bar{y}) \subseteq -\{0\}\times  \bigg( \partial \Psi_e(0)\cap N(\bar{y},\mathcal{H}_F(\bar{x}))\bigg).
\end{equation}

The estimate is then obtained by replacing the term $\partial(- f)(\bar{x},\bar{y})$ in \eqref{eq: f_u convex aux} by the upper estimate obtained in \eqref{eq: aux f_u 2}.
\end{proof}

\section{Optimality conditions for Set Optimization problems} \label{sec: opt cond}

In this section we will obtain optimality conditions for set optimization problems based on our previous results. We start by formally defining the set optimization problem and the solution concepts that will be considered. 

\begin{Definition}\label{def: minimal solutions setopt}
 Let Assumption \ref{ass 1}  be fulfilled and let $r \in \{l,u\}.$ The set optimization problem is defined as  
\begin{equation}\label{eq:SP}
		\min\limits_{x\in \Omega}   \quad F(x), \tag{$\mathcal{SOP}$}
\end{equation}
and its minimal solutions are understood in the following sense: we say that $\bar{x}\in \Omega$ is a  
\begin{enumerate}
\item $\preceq^{(r)}_K$-weakly minimal solution of \eqref{eq:SP} if 

$$\nexists \; x\in \Omega\setminus\{\bar{x}\}: F(x)\prec^{(r)}_K F(\bar{x}).$$

%\item $\preceq^{(r)}_K$-minimal solution of \eqref{eq:SP} if for every $x\in \Omega$ the following implication holds:
%
%$$F(x)\preceq^{(r)}_{K} F(\bar{x})\Longrightarrow F(\bar{x})\preceq^{(r)}_K F(x).$$
%

\item $\preceq^{(r)}_K$-strictly minimal solution of \eqref{eq:SP} if 

$$\nexists \; x\in \Omega\setminus\{\bar{x}\}: F(x)\preceq^{(r)}_{K} F(\bar{x}).$$

\item weakly minimal solution of \eqref{eq:SP} if 

$$\exists\; \bar{y} \in F(\bar{x}): \; F(\Omega)\cap \left(\bar{y}- \Int K\right) = \emptyset.$$
\end{enumerate} 
If in the above definition we replace $\Omega$ by $\Omega\cap U$, with $U$ being a neighborhood of $\bar{x},$ we say that $\bar{x}$ is a local ($\preceq^{(r)}_K$-weakly,  $\preceq^{(r)}_K$-, $\preceq^{(r)}_K$- strictly, weakly)minimal solution respectively.
\end{Definition}

\begin{Remark} It is easy to see that $\preceq^{(r)}_K$- strictly minimal solutions are $\preceq^{(r)}_K$- weakly minimal. In addition, the minimality concept in Definition \ref{def: minimal solutions setopt} $(iii)$ is the one used in the vector approach for set optimization problems \cite{KTZ}. It is known \cite[Proposition 2.10]{HR2007} that weakly minimal solutions of \ref{eq:SP} are also $\preceq_K^{(l)}$- weakly minimal in a slightly different sense. A similar statement can be made about the set relation $\preceq_K^{(u)},$ see also \cite[Remark 2.11]{HR2007} Conversely, it was proved in \cite{KKY2017} that, if $F(\bar{x})$ has a strongly minimal element and $\bar{x} $ is a $\preceq_K^{(l)}$- weakly minimal solution of \eqref{eq:SP}, then $\bar{x}$ is also a weakly minimal solution.
%
%\item In the literature, there is another concept of $\preceq_K^{(r)}$- weak minimality that differs from the one we introduced in Definition \ref{def: minimal solutions setopt} $(i),$ see \er{Reference}. There, $\bar{x}$ is a $\preceq_K^{(r)}$- weakly minimal solution of \eqref{eq:SP} if 
%\begin{equation}\label{eq: wminimality other}
%\forall \; x\in \Omega:\; F(x) \prec_K^{(r)} F(\bar{x}) \Longrightarrow F(\bar{x}) \prec_K^{(r)} F(x).
%\end{equation} It is easy to see that if $\bar{x}$ is $\preceq_K^{(r)}$- weakly minimal in our sense, so is in the sense of \eqref{eq: wminimality other}. The converse is true if $\WMin(F(\bar{x}),K)\neq$

\end{Remark}

Of course, global solutions of \eqref{eq:SP} are also local solutions. Our next proposition confirms that, as in the scalar case, the converse holds under convexity. 

\begin{Proposition}
Let Assumption \ref{ass 1}  be fulfilled and fix $r \in \{l,u\}.$  Suppose that $\Omega$ is convex,  that $F$ is $\preceq_K^{(r)}$-convex and that $\bar{x}$ is a local $\preceq_K^{(r)}$-weakly minimal solution of \eqref{eq:SP}. The following statements are true:
\begin{enumerate}
\item If $r = l,$  then $\bar{x}$ is also a global $\preceq_K^{(l)}$-weakly minimal solution.

\item If $r = u$ and $\mathcal{H}_F(\bar{x})$ is convex, then $\bar{x}$ is also a global $\preceq_K^{(u)}$-weakly minimal solution.
\end{enumerate}
\end{Proposition}

\begin{proof}
Since the proofs are similar and resemble the one in the scalar case, we only show $(ii).$ See also \cite[Proposition 5]{ha2019} for a proof of $(i)$ with a slightly different optimality concept. Let $U$ be the neighborhood of $\bar{x}$ such that 
$$\forall \; x\in \Omega \cap U\setminus \{\bar{x}\}:\; F(x)\nprec^{(r)}_K F(\bar{x})$$ and suppose that $\bar{x}$ is not a global $\preceq_K^{(u)}$-weakly minimal solution of \eqref{eq:SP}. Then, we can find $\tilde{x} \in \Omega \setminus \{\bar{x}\}$ such that $F(\tilde{x}) \prec_K^{(u)} F(\bar{x}).$ Hence, we get the existence of $\lambda \in (0,1]$ such that $x_\lambda: = \lambda \tilde{x} +(1- \lambda)\bar{x} \in \Omega\cap U \setminus \{\bar{x}\}.$  It follows that

\begin{eqnarray*}
F(x_\lambda) & \subseteq & F(x_\lambda)-K\\
             & \overset{(F \textrm{ is } \preceq_K^{(u)}-\textrm{convex})}{\subseteq} & \lambda F(\tilde{x}) +(1- \lambda) F(\bar{x})-K\\
             & \subseteq & \lambda \mathcal{H}_F(\tilde{x}) +(1- \lambda) \mathcal{H}_F(\bar{x}) - K\\
             & \overset{(\textrm{as } F(\tilde{x})\prec_K^{(u)} F(\bar{x}))}{\subseteq} & \lambda \mathcal{H}_F(\bar{x}) +(1- \lambda) \mathcal{H}_F(\bar{x}) - \Int K\\
             & \overset{ (\mathcal{H}_F(\bar{x}) \textrm{ is convex})}{=}&  \mathcal{H}_F(\bar{x}) - \Int K\\ 
             & = & F(\bar{x}) - \Int K,
\end{eqnarray*} which is equivalent to $F(x_\lambda) \prec_K^{(u)} F(\bar{x}).$ This contradicts to the local minimality of $F$ at $\bar{x}.$

\end{proof}

In the following theorem we establish relationships between the set-valued problem and a corresponding scalar problem. We want to mention that a similar statement to $(i)$ below have been established in \cite[Corollary 4.11]{HR2007}  for the case $r = l.$

\begin{Theorem}\label{thm: scalar theorem}
Let Assumption \ref{ass 1}  be fulfilled and, for $r\in \{l,u\},$ consider the functional $f_{r,\bar{x}}$ in Definition \ref{def- scalarization} $(iii)$. The following assertions are true:

\begin{enumerate}

\item If $\bar{x}$ is a local $\preceq_K^{(r)}$-weakly minimal solution of \eqref{eq:SP}, then $\bar{x}$ is a local solution of  the problem 
\begin{equation}\label{eq:Pr}
\min\limits_{x\in \Omega} \; f_{r,\bar{x}}(x).\tag {$P_r$}
\end{equation} 
\item Conversely, suppose that $\bar{x}$ is a local strict solution of problem \eqref{eq:Pr} and either $r=l$ and $\WMin(F(\bar{x}),K)$ $\neq \emptyset,$ or $r=u$ and $\WMax(F(\bar{x}),K)\neq \emptyset.$  Then, $\bar{x}$ is a local $\preceq_K^{(r)}$-strictly minimal solution of \eqref{eq:SP}.

\end{enumerate}
\end{Theorem}

\begin{proof}
$(i)$ Assume that $\bar{x}$ is not a local solution of \eqref{eq:Pr}. Then, for every neighborhood $U$ of $\bar{x}$ we can find $\tilde{x} \in \Omega \cap U$ such that  

\begin{equation}\label{eq: f_r<0}
f_{r,\bar{x}}(\tilde{x})< f_{r,\bar{x}}(\bar{x})\leq 0.
\end{equation} We just analyze the case $r = u$ since the other one is similar. From the definition of $f_{u,\bar{x}}$ and \eqref{eq: f_r<0}, we deduce that for every $ \tilde{y}\in F(\tilde{x}),$ the inequality $ g_{u,\bar{x}}(\tilde{y})<0$ holds. Equivalently, we obtain

$$\forall \; \tilde{y}\in F(\tilde{x})\;\exists\, \bar{y} \in F(\bar{x}): \Psi_e(\tilde{y}- \bar{y})<0.$$ Again, by Proposition \ref{prop: properties of tammer function} $(vi),$ we obtain $F(\tilde{x})\prec^{(u)}_{ K} F(\bar{x}),$ a contradiction.

%\begin{itemize}
%\item \textbf{Case 1:} $r=l$\\
%In this case, it follows from \eqref{eq: f_r<0} and the definition of $f_{l,\bar{x}}$ that, for every $\bar{y}\in F(\bar{x}),$ the inequality $ g_l(\tilde{x},\bar{y})<0$ holds.  Equivalently, we get,
%$$\forall \;\bar{y}\in F(\bar{x}),\;\exists \;\tilde{y}\in F(\tilde{x}): \Psi_e(\tilde{y}- \bar{y})<0.$$ This, together with Proposition \ref{prop: properties of tammer function} $(vi),$ gives us $F(\tilde{x})\prec^{(l)}_{ K} F(\bar{x}),$ a contradiction.
%
%\item \textbf{Case 2:} $r=u$\\
%
%
%\end{itemize}

$(ii)$ By Proposition \ref{prop: g =0} $(iii)$ we know that $f_{r,\bar{x}}(\bar{x})$ is finite. Assume that $\bar{x}$ is not a local $\preceq_K^{(r)}$-strictly minimal solution of \eqref{eq:SP}. Then, for any neighborhood $U$ of $\bar{x}$ we can find $\tilde{x}\in (\Omega\cap U)\setminus \{\bar{x}\}$ such that $$F(\tilde{x})\preceq^{(r)}_K F(\bar{x}).$$ Hence, according to  Theorem \ref{thm: characterization of set relations by scalarizing functionals}, we get 

$$f_{r,\bar{x}}(\tilde{x}) \leq  f_{r,\bar{x}}(\bar{x}).$$ This contradicts the fact that $\bar{x}$ is a local strict solution of \eqref{eq:Pr}. 
\end{proof}

Necessary optimality conditions for \eqref{eq:SP} with respect to the relation $\preceq_K^{(l)}$ are established in the next theorem.

\begin{Theorem}\label{thm: opt cond l}
Let Assumption \ref{ass 1}  be fulfilled and suppose that $\bar{x}$ is a local $\preceq_K^{(l)}$-weakly minimal solution of \eqref{eq:SP}. The following statements are true:

\begin{enumerate}

\item Suppose that $\Omega$ is convex, that $F$ is $\preceq^{(l)}_K$-convex and locally $l$-bounded at $\bar{x},$ and that $F(\bar{x})$ is strongly $K$- compact. Then, 

\begin{equation}\label{eq: opt conf convex l}
0 \in \overline{\operatorname{conv}}^*\left(\bigcup_{\bar{y} \in \Min(F(\bar{x}),K)} D^* \mathcal{E}_F(\bar{x},\bar{y})\left[\partial \Psi_e(0)\right]\right) + N(\bar{x},\Omega).
\end{equation}

This condition is sufficient for optimality provided that, in addition, $F$ is strongly $K$- compact valued in $\Omega.$

\item Suppose that $X$ and $Y$ are Asplund spaces, that $F$ is locally Lipschitzian at $\bar{x},$ and that the rest of the conditions in Theorem \ref{thm: upper estimate subdiff f_l} are fulfilled. Then,

\begin{equation}\label{eq: opt cond lip f_l}
0\in \overline{\operatorname{conv}}^*\left( \bigcup_{\bar{y}\in \WMin(F(\bar{x}),K)} \Bigg\{x^*\in X^*: \exists \;  y^* \in N(\bar{y},F(\bar{x})):  (x^*,y^*) \in G_{(\bar{x},\bar{y})} \Bigg\}\right)+N(\bar{x},\Omega),
\end{equation} where

$$G_{(\bar{x},\bar{y})}= \overline{\operatorname{conv}}^*\left( \bigcup_{\underset{z^* \in \partial \Psi_e(\bar{z}-\bar{y})}{\bar{z}\in F(\bar{x})\cap (\bar{y}-\bd K)}} D^*F(\bar{x},\bar{z})(z^*) \times \{-z^*\} \right).$$

\end{enumerate}

\end{Theorem}
\begin{proof} By Theorem \ref{thm: scalar theorem}, it follows that $\bar{x}$ is a solution of 
\begin{equation}\label{eq:Pl}
\min\limits_{x\in \Omega} \; f_{l,\bar{x}}(x).\tag {$P_l$}
\end{equation}

$(i)$ Because of Theorem \ref{thm: convexity of scalarizations } $(i)$, we know that $f_{l,\bar{x}}$ is convex and continuous at $\bar{x}. $ The classical necessary and sufficient condition for convex problems  \cite[Proposition 5.1.1]{Schirotzek2007} is now read as $0\in \partial f_{l,\bar{x}}(\bar{x})+ N(\bar{x},\Omega).$ Hence, the first part of the statement follows from Theorem \ref{thm: subdiff lower l scalar}.

Suppose now that $F$ is strongly $K$- compact valued in $\Omega$ and that $\bar{x}$ is not a $\preceq_K^{(l)}$- weakly minimal solution of \eqref{eq:SP}. Then, without loss of generality we can assume that $F$ is compact valued and that there exists $\tilde{x}\in \Omega$ such that 

\begin{equation}\label{eq: ftilde less fbar}
F(\tilde{x}) \prec_K^{(l)}F(\bar{x}).
\end{equation}

We claim that $f_{l,\bar{x}}(\tilde{x})<0 = f_{l,\bar{x}}(\bar{x}),$ which contradicts  \eqref{eq: opt conf convex l}. Indeed, note that because $F(\tilde{x})$ is compact, the functional $g_l(\tilde{x},\cdot)$ is finite. It is also upper semicontinuous in $Y$ because it is the infimum of continuous functionals. Since $F(\bar{x})$ is compact, the classical Weierstrass's theorem tells us that the problem 

$$\max \limits_{z \in F(\bar{x})}  \; g_l(\tilde{x},z) $$ has a solution $\bar{y}.$ According to \eqref{eq: ftilde less fbar}, we can find $\tilde{y} \in F(\tilde{x})$ such that $\tilde{y} \prec_K \bar{y}.$ Hence, we get

$$f_{l,\bar{x}}(\tilde{x}) = g_l(\tilde{x},\bar{y}) \leq \Psi_e(\tilde{y}- \bar{y}) <0,$$ as desired.

$(ii)$ Similarly to the previous case, by Theorem \ref{thm: Lipschitz of scalarizations} $(i)$ we obtain that $f_{l,\bar{x}}$ is locally Lipschitzian at $\bar{x}.$ Hence, all the assumptions for the necessary optimality conditions in \cite[Proposition 5.3]{Mordukhovich2} are satisfied. From this we get  $0\in \partial f_{l,\bar{x}}(\bar{x})+ N(\bar{x},\Omega).$ The result follows then from Theorem \ref{thm: upper estimate subdiff f_l}.
\end{proof}

With a similar argument to the one  in the previous theorem, we can obtain the optimality conditions for problems with the relation $\preceq_K^{(u)}$. The proof is hence omitted.
\begin{Theorem}\label{thm: opt cond u}
In addition to Assumption \ref{ass 1}, suppose that $X$ and $Y$ are Asplund spaces and that $\bar{x}$ is a local $\preceq_K^{(u)}$- weakly minimal solution of \eqref{eq:SP}. The following statements are true:

\begin{enumerate}

\item Suppose that  $F$ is $\preceq^{(u)}_K$-convex and locally $u$-upper bounded at $\bar{x},$ and that the conditions in Theorem \ref{thm: subdif f_u convex case} are fulfilled. Then,

$$0 \in   - \overline{\operatorname{conv}}^*\left( \bigcup_{\bar{y}\in \WMax(\mathcal{H}_F(\bar{x}),K)}  D^*\mathcal{H}_F(\bar{x},\bar{y})\left[-\partial \Psi_e(0) \cap N(\bar{y},\mathcal{H}_F(\bar{x}))   \right] \right) + N(\bar{x}, \Omega).$$

\item Suppose that the $F$ is locally Lipschitzian at $\bar{x}$ and that the conditions in Theorem \ref{thm: subdif f_u lip case} are fulfilled. Then,

\begin{equation}\label{eq: opt cond lip f_u}
0\in -\overline{\operatorname{conv}}^*\left( \bigcup_{\bar{y}\in \WMax(F(\bar{x}),K)}  D^*F(\bar{x},\bar{y})\left[H_{(\bar{x},\bar{y})} \right] \right)+ N(\bar{x},\Omega),
\end{equation} where 

$$H_{(\bar{x},\bar{y})}: =  -\overline{\operatorname{conv}}^*\left( \bigcup_{\bar{z} \in F(\bar{x})\cap (\bar{y}+ \bd K)} \partial \Psi_e(\bar{y}-\bar{z}) \cap N \left(\bar{z},F(\bar{x})\right)\right).$$

\end{enumerate}

\end{Theorem}

Theorem \ref{thm: opt cond l} and Theorem \ref{thm: opt cond u} motivates the following definition. 
\begin{Definition}\label{def: stationary points}
Let Assumption \ref{ass 1}  be fulfilled. We say that $\bar{x}$ is a

\begin{enumerate}
\item $\preceq_K^{(l)}$- stationary point of \eqref{eq:SP}, if \eqref{eq: opt cond lip f_l} is fulfilled,
\item $\preceq_K^{(u)}$- stationary point of \eqref{eq:SP}, if \eqref{eq: opt cond lip f_u} is fulfilled.
\end{enumerate} 

\end{Definition}

We conclude this section with the following example, that illustrate our results and compare them with other results obtained for the vector approach.

\begin{Example}\label{ex: master example}

Let $X= \Omega= \Bbb R,\, Y=\Bbb R^2, \; K= \R_+^2, \; e= \begin{pmatrix} 1 \\ 1 \end{pmatrix},\; \bar{x}=0.$  Consider the function $f: \Bbb R \to \Bbb R^2$ and the set-valued mapping $F: \Bbb R \rightrightarrows \Bbb R^2$ defined respectively as

$$
f(x):=  \begin{pmatrix} x+1 \\ x-1 \end{pmatrix},\;\; F(x):=\{f(x),-f(x)\}.
$$ In particular, we have $\nabla f(\bar{x}) = (1 \; 1).$ Then:

\begin{enumerate}
\item $F$ is locally Lipschitzian at $\bar{x}.$ 

\item $\bar{x}$ is both a local $\preceq_K^{(l)}, \preceq_K^{(u)}$- weakly minimal solution of \eqref{eq:SP}:

Indeed, it is easy to verify that, choosing $U = (-1,1):$ 

$$\forall \; x \in U: \;F(x) \nprec_K^{(l)} F(\bar{x}), \quad F(x) \nprec_K^{(u)} F(\bar{x}).$$

\item  $\bar{x}$ is not a local weakly minimal nor a local weakly maximal solution with the vector approach:

Indeed, note that in any neighborhood $U$ of $\bar{x}$ we can find $x \in U\setminus\{\bar{x}\}$ such that $-x \in U.$ Then, it is easy to check that 

$$F(\bar{x}) \subset \big( F(x)+ \Int K \big) \cup \big( F(-x)+ \Int K\big),$$

$$F(\bar{x}) \subset \big( F(x)- \Int K \big) \cup \big( F(-x)- \Int K\big).$$

\item $\bar{x}$ is not a stationary point in the sense of the vector approach:

Since $f(\bar{x}) \neq -f(\bar{x})$ and $\gph F = \gph f \cup \gph(-f),$ we have that $\gph F = \gph f $  and $\gph F = \gph(-f) $ around $(\bar{x},f(\bar{x}))$ and $(\bar{x},-f(\bar{x}))$ respectively. By the differentiability of $f$ and Remark \ref{rem: coderivative of funct}, we obtain

\begin{equation}\label{eq: coder funct1}
\forall \; z^*\in \Bbb R^2:\; D^* F(\bar{x},f(\bar{x}))(z^*) = \{\nabla f(\bar{x}) z^*\}= \{z_1^*+z_2^*\},
\end{equation}

\begin{equation}\label{eq: coder funct2}
\forall \; z^*\in \Bbb R^2:\; D^* F(\bar{x},-f(\bar{x}))(z^*) = \{-\nabla f(\bar{x}) z^*\}= \{-(z_1^*+z_2^*)\}.
\end{equation} Now, we recall that $\bar{x}$ is a stationary point of $F$ in the sense of the vector approach (see \cite[Theorem 5.1]{Bao10} and \cite[Theorem 3.11]{dureastrugariu2011} ) if there exists $\bar{y} \in F(\bar{x})$ and $y^* \in K^* \setminus\{0\}$ such that 

$$0 \in D^* F(\bar{x},\bar{y})(y^*).$$

Since $K^*= K$ in our context, it is then easy to check that  $$0 \in D^* F(\bar{x},f(\bar{x}))(y^*),\; y^*\in K^*\Longleftrightarrow y^*= \begin{pmatrix} 0 \\ 0 \end{pmatrix}.$$ Similarly, we obtain that 

$$0 \in D^* F(\bar{x},-f(\bar{x}))(y^*),\; y^*\in K^*\Longleftrightarrow y^*=\begin{pmatrix} 0 \\ 0 \end{pmatrix}.$$ It follows that $\bar{x}$ is not a  stationary point in the sense of the vector approach.

\item $\bar{x}$ is both $\preceq_K^{(l)}$- and $\preceq_K^{(u)}$-stationary:

Of course, this is a direct consequence of Theorem \ref{thm: opt cond l} and Theorem \ref{thm: opt cond u}, but we show the calculus for completeness. First, we note that $\WMin(F(\bar{x}),K)= \WMax(F(\bar{x}),K) =  F(\bar{x}).$  Because $F(\bar{x})$ consists of isolated points, we obtain 

\begin{equation}\label{eq: Normal isolated}
N(f(\bar{x}),F(\bar{x}))= N(-f(\bar{x}),F(\bar{x}))= \Bbb R^2.
\end{equation} On the other hand, from Proposition \ref{prop: properties of tammer function} $(v)$ we have

\begin{equation}\label{eq: example subdiff tammer}
\partial \Psi_e(0)= \{k^* \in \R_+^2 : k_1^*+k_2^* = 1\}.
\end{equation}

The $\preceq_K^{(l)}$- stationarity of $\bar{x}$ is now equivalent to $0 \in \overline{\operatorname{conv}}(A_1 \cup A_2),$ where
\begin{equation}
A_1 :=   \left\{x^* \in \Bbb R: \exists \; y^* \in \Bbb R^2: (x^*, {y^*})^T \in G_{(\bar{x},f(\bar{x}))}\right\},
\end{equation}

\begin{equation}
A_2 :=   \left\{x^* \in \Bbb R: \exists \; y^* \in \Bbb R^2: (x^*, {y^*})^T \in G_{(\bar{x},-f(\bar{x}))}\right\}.    
\end{equation} 

We have 

\begin{eqnarray*}
G_{(\bar{x},f(\bar{x}))} & \overset{\eqref{eq: coder funct1}}{=} & \overline{\operatorname{conv}}\left(\bigcup_{z^* \in \partial \Psi_e(0)} \{z_1^*+ z_2^*\} \times \{-z^*\}\right)\\
                         & \overset{\eqref{eq: example subdiff tammer}}{=} & \overline{\operatorname{conv}}\left(\bigcup_{z^* \in \partial \Psi_e(0)} \{1\} \times \{-z^*\}\right)\\
                         & = & \{1\} \times \left(-\partial \Psi_e(0)\right).
\end{eqnarray*} From this, we deduce that $A_1 = \{1\}.$ Using a similar  argument we can obtain $G_{(\bar{x},-f(\bar{x}))}= \{-1\} \times \left(-\partial \Psi_e(0)\right),$ from which we obtain $A_2 = \{-1\}.$ Hence, we have $$0 \in [-1,1]=\overline{\operatorname{conv}}(A_1 \cup A_2),$$ and the $\preceq_K^{(l)}$- stationarity of $\bar{x}$ follows.

Next, we show that $\bar{x}$ is also $\preceq_K^{(u)}$- stationary. This is equivalent to $0 \in \overline{\operatorname{conv}}(B_1 \cup B_2),$ where
\begin{equation}\label{eq: b1}
B_1 := - D^*F(\bar{x},f(\bar{x})) \left[ H_{(\bar{x},f(\bar{x}))}\right],
\end{equation} 

\begin{equation}\label{eq: b2}
B_2 := - D^*F(\bar{x},f(\bar{x})) \left[ H_{(\bar{x},-f(\bar{x}))}\right].   
\end{equation} 

In this case we have
\begin{eqnarray*}
 H_{(\bar{x},f(\bar{x}))}& = & -\overline{\operatorname{conv}}\left( \partial \Psi_e(0) \cap N \left(f(\bar{x}),F(\bar{x})\right)\right)\\
                         & \overset{\eqref{eq: Normal isolated}}{=}& - \partial \Psi_e(0).
\end{eqnarray*} From this, we deduce that $$B_1 \overset{\eqref{eq: b1} }{=} -D^*F(\bar{x},f(\bar{x}))[- \partial \Psi_e(0)] \overset{\eqref{eq: coder funct1}}{=} \{1\}.$$ Similarly, we can obtain $H_{(\bar{x},-f(\bar{x}))} =  - \partial \Psi_e(0),$ from which we get $$B_2 \overset{\eqref{eq: b2} }{=} -D^*F(\bar{x},-f(\bar{x}))[- \partial \Psi_e(0)] \overset{\eqref{eq: coder funct2}}{=} \{-1\}.$$ Hence, we have $0 \in [-1,1] = \overline{\operatorname{conv}}(B_1 \cup B_2),$ and $\bar{x}$ is $\preceq_K^{(u)}$- stationary.

\end{enumerate}

\end{Example}

%\begin{Example}
%
%Let $Y = \R^m,\; K= \R^m_+,\; e = (1 \;1\; \ldots \;1)^T,$ and $f:X \rightarrow \R^m$ be locally Lipschitzian at $\bar{x}.$ Consider the set-valued mapping $F:X \rightrightarrows \R$ given by $F(x):= \{f(x)\}.$ Let us analyze what is the estimate \eqref{eq: upper estimate subdiff f_l} in this case.
%
%It is easy to see that $\WMin(F(\bar{x}),K) = \{f(\bar{x})\}$ and that $N(f(\bar{x}),F(\bar{x})) = \R^m.$ We also have
%
%
%
%\end{Example}

\section{Conclusions}\label{sec: conclussions}

In this paper, we considered the set optimization problem with respect to the lower and upper less relations. The main contributions  are the optimality conditions in Theorem \ref{thm: opt cond l} and Theorem \ref{thm: opt cond u}, that are derived under the Lipschitzianity of the set-valued objective mapping and other natural assumptions. Perhaps the most attractive feature of our necessary conditions is that we do not require neither convexity nor compactness of the images of $F,$ nor the existence of a strongly minimal element in the optimal set, which are some of the drawbacks of the other approaches in the literature \cite{Alonsomarin2005,AlonsoMarin2009,dempepilecka2016, ha2018, ha2019, HueJimNov2020, jahn2015, Jahn2017,KKY2017,kongetal2017,oussarhandaidai2018, pilecka2014, rodmarinsama2007}. 

The results obtained also open several ideas for further research. In particular, the scheme employed could be easily extended to other set relations, like those described in \cite{jahnha2011,karamanemrahetal2018,kuroiwa1998, Kuroiwa2001} and that were not mentioned here. In addition,  it is also of interest to relax the Lipschitzian assumption, maybe replacing it by some type of lower semicontinuity property, and therefore obtaining stronger results. Finally, we believe that our optimality conditions are the first step towards deriving algorithms for set optimization problems that converge to stationary points.

\section*{Aknowledgments}\label{sec: aknowledgments}

The authors would like to thank Boris Mordukhovich (Wayne State University) and Truong Q. Bao (Northern Michigan University) for comments and suggestions that improved an earlier version of this paper.

%%%%%%%%%%%%%%%%%%
%% Bibliography %%
%%%%%%%%%%%%%%%%%%
\bibliographystyle{siam}
\bibliography{references}
\end{document}